\documentclass[12pt]{amsart}

\usepackage[utf8]{inputenc}

\usepackage[margin=1in]{geometry}

\usepackage{xr-hyper}
\usepackage{hyperref}

\usepackage{bfc}
\usepackage{mfdmetrics}

\numberwithin{theorem}{section}
\numberwithin{equation}{section}

\date{\today}
\title{The Riemannian $L^2$ topology on the manifold of Riemannian metrics}
\author{Brian Clarke}
\thanks{This research was supported by NSF grant DMS-0902674.}
\address{Department of Mathematics, Stanford University, Stanford, CA
  94305-2125}
\email{\href{mailto:bfclarke@math.stanford.edu}{bfclarke@math.stanford.edu}}
\urladdr{\href{http://math.stanford.edu/~bfclarke/}{http://math.stanford.edu/${\sim}$bfclarke/}}

\begin{document}

\maketitle{}

\begin{abstract}
  We study the manifold of all Riemannian metrics over a closed,
  finite-di\-men\-sion\-al manifold.  In particular, we investigate
  the topology on the manifold of metrics induced by the distance
  function of the $L^2$ Riemannian metric---so called because it
  induces an $L^2$ topology on each tangent space.  It turns out that
  this topology on the tangent spaces gives rise to an $L^1$-type
  topology on the manifold of metrics itself.  We study this new
  topology and its completion, which agrees homeomorphically with the
  completion of the $L^2$ metric.  We also give a user-friendly
  criterion for convergence (with respect to the $L^2$ metric) in the
  manifold of metrics.
\end{abstract}

\section{Introduction}\label{sec:introduction}

Let $M$ be a smooth, closed, oriented, finite-dimensional manifold,
and let $\M$ denote the space of all Riemannian metrics on $M$.  The
space $\M$ is naturally a Fréchet manifold, and it possesses a
canonical $L^2$ Riemannian metric.  (So called because it induces an
$L^2$-type scalar product on the \emph{tangent spaces} of $\M$.)
Despite the fact that $\M$ is a contractible space, the $L^2$ metric
has rich local geometry---for instance, its curvature is nonnegative,
and its geodesics are explicitly computable
\cite{freed89:_basic_geomet_of_manif_of,gil-medrano91:_rieman_manif_of_all_rieman_metric}.  The $L^2$ metric
has arisen in Teichmüller theory \cite{tromba-teichmueller}, as well
as in studies of the moduli space of Riemannian metrics
\cite{ebin70:_manif_of_rieman_metric}.

In \cite{clarke:_compl_of_manif_of_rieman_metric} and
\cite{clarke:_metric_geomet_of_manif_of_rieman_metric}, we made steps
towards understanding the global geometry of the $L^2$ metric.  In
particular, we showed that the $L^2$ Riemannian metric induces a
metric space structure on $\M$.  (See Section
\ref{sec:manifold-metrics} for a discussion of why this result is
nontrivial.)  We also gave a natural identification of the completion
of $\M$---with respect to the $L^2$ metric---as a quotient space of
the space of all measurable, finite-volume Riemannian semimetrics on
$M$ (Theorem \ref{thm:6}).

In this paper, we carry this study one step further by giving a
simplified description of the topology induced by the distance
function $d$ of the $L^2$ metric on $\M$, as well as on the completion
$\overline{\M}$.  Interestingly, the $L^2$ topology on the tangent
spaces of $\M$ translates---via $d$---into an $L^1$-type topology on
$\M$ and $\overline{\M}$.  An analogous result can be seen in
\cite[Thm.~5.2]{clarked):_compl_of_manif_of_rieman}, where the $L^2$
metric induces an $L^{n/2}$-type topology on the space of metrics
conformally equivalent to a given metric (where $n$ denotes the
dimension of $M$.)

We can describe the main result of the paper as follows.  Let $\satx$
denote the set of symmetric $(0, 2)$-tensors based at $x \in M$, and
let $\Matx \subset \satx$ be those tensors that induce a
positive-definite scalar product on $T_x M$.  (Thus, $\M$ is given by
the smooth sections of the bundle $\cup_{x \in M} \Matx$.)  We define
a quotient space by
\begin{equation*}
  \overline{\Matx} := \grpquot{\cl(\Matx)}{\partial \Matx},
\end{equation*}
where $\cl(\Matx)$ denotes the closure of $\Matx \subset \satx$ and
$\partial \Matx$ denotes the boundary of $\Matx$.  In other words,
$\overline{\Matx}$ is given by positive \emph{semi}-definite tensors
at $x$, where we identify all tensors that are not positive definite.
We will see later (Theorem \ref{thm:6}) that in a precise sense, the
completion $\overline{(\M, d)}$ is given by the measurable sections of
the bundle $\cup_{x \in M} \overline{\Matx}$ that have finite total
volume.

Let $g \in \M$ be an arbitrary reference metric.  Then we have the
following result.

\begin{theorem*}[Theorem \ref{thm:13}]
  For each $x \in M$, there exists a metric (in the sense of metric
  spaces) $\tgx$ on $\overline{\Matx}$ such that the topology of
  $\overline{(\M, d)}$ agrees with the $L^1$ topology of $\tgx$.  That
  is, for $g_0, g_1 \in \overline{(\M, d)}$, let
  \begin{equation*}
    \TM(g_0, g_1) := \integral{M}{}{\tgx(g_0(x), g_1(x))}{d \mu_g},
  \end{equation*}
  where $\mu_g$ denotes the volume form induced by $g$.  Then the
  topology induced by the metric $\TM$ agrees with the topology of
  $\overline{(\M, d)}$, and $\overline{(\M, d)}$ is complete with
  respect to $\TM$.
\end{theorem*}

When studying the topology of $d$, it is desirable to swap this
Riemannian distance function for the simpler description of the above
theorem in terms of the $L^1$ topology of a bundle of metric spaces
over $M$.  (See Sections \ref{sec:manifold-metrics} and
\ref{sec:motiv-defin}.)  In particular, calculating or estimating
$\TM$ involves first computing with $\tgx$ (a finite-dimensional
problem) for each $x$, and then integrating the results over $M$.  On
the other hand, calculating or estimating $d$ involves considering
infima of lengths of paths in $\M$ with respect to the $L^2$
Riemannian metric, a decidedly infinite-dimensional problem.  We will
give some examples of the utility of this approach in Section
\ref{sec:anoth-char-conv}, where we show the discontinuity of numerous
geometric quantities on $M$ with respect to $d$.

The eventual goal of this effort is an understanding of the structure
induced by $d$ on the moduli space of Riemannian metrics (sometimes
also called \emph{superspace}).  This is the quotient space $\M / \D$,
where $\D$ denotes the diffeomorphism group of $M$, acting on $\M$ by
pull-back.  Since $\D$ acts by isometries on $\M$ with the $L^2$
metric \cite[\S 6.1.2]{clarked):_compl_of_manif_of_rieman}, $d$
induces a \emph{pseudo}metric-space structure on $\M / \D$ (which not
a manifold, but rather a stratified space
\cite{bourguignon75:_une_strat_de_lespac_des_struc_rieman}).  It would
be interesting to see what our results can say about the completion
and metric geometry of $\M / \D$, but the first question one must ask
is whether $\M / \D$ is a metric space with this pseudometric.  This
seems to be a difficult question---see the discussion and examples
following Theorem \ref{thm:1} for more on this.  Nevertheless, we hope
that the theorem quoted above may give us some more useful tools for
studying these issues in future papers.

This paper is organized as follows.  In Section
\ref{sec:defin-prev-results}, we review the definitions and previous
results that we will require.  This includes a discussion of the
fundamentals on the manifold of metrics and the $L^2$ metric, the
completion of $\M$, and some structures and properties that were laid
out in our previous works
\cite{clarke:_compl_of_manif_of_rieman_metric} and
\cite{clarke:_metric_geomet_of_manif_of_rieman_metric}.  We also
include a few novel results and extensions of previous results that
will be useful to us in subsequent sections.

In Section \ref{sec:metric-tm}, we include a detailed discussion of
the metric $\TM$ given in the above theorem.  In particular, we will
examine the relationship between $\TM$ and volume, as well as
describing the complete $L^1$ space determined by $\tgx$ on the bundle
$\cup_{x \in M} \overline{\Matx}$.

Finally, in Section \ref{sec:convergence-mf}, we give the proof of the
above-quoted main theorem.  In addition, at the end of the section we
give an alternative characterization of convergence with respect to
$d$ when the limit is an element of $\M$ (as opposed to the
completion).  This in fact gives a relatively easy to verify criterion
for convergence---it is simply a kind of convergence in measure,
together with a strong convergence of the volume forms.

\subsection*{Acknowledgements}

I wish to thank Guy Buss for helpful discussions during the
preparation of this paper.  This paper is an extension of ideas from
my Ph.D.~thesis, and so I also wish to thank my advisor Jürgen Jost
for introducing the topic to me and for his years of support.

\section{The manifold of Riemannian
  metrics}\label{sec:defin-prev-results}

In this section, we review some of the structures and results relating
to the geometry of the $L^2$ metric's distance function that were
introduced in \cite{clarke:_compl_of_manif_of_rieman_metric} and
\cite{clarke:_metric_geomet_of_manif_of_rieman_metric}.  We will draw
on these results throughout the rest of the paper.  In addition, we
introduce a few new concepts and some new notation that will be
important for us.  Most of the facts stated here can be found in
greater detail in \cite[\S 2]{clarke:_compl_of_manif_of_rieman_metric}
and \cite[\S 2]{clarke:_metric_geomet_of_manif_of_rieman_metric}.  (An
even more elementary discussion can be found in
\cite[Ch.~2]{clarked):_compl_of_manif_of_rieman}.)

To begin, though, we must recall some fundamentals on the manifold of
Riemannian metrics and the $L^2$ metric.

\subsection{The manifold of metrics}\label{sec:manifold-metrics}

The facts stated in this section were established in
\cite{ebin70:_manif_of_rieman_metric},
\cite{freed89:_basic_geomet_of_manif_of} and
\cite{gil-medrano91:_rieman_manif_of_all_rieman_metric}, to which we
refer for further details.  A detailed overview is also given in
\cite[Ch.~2]{clarked):_compl_of_manif_of_rieman}.

Let $M$ be a $C^\infty$-smooth, closed, oriented manifold of dimension
$n$.  We denote by $\s$ the vector space of smooth, symmetric $(0,
2)$-tensor fields on $M$.  It is a Fréchet space when equipped with
the family of Sobolev $H^s$ norms for $s \in \N$ \cite[\S
2.5.1]{clarked):_compl_of_manif_of_rieman}.

The set $\M$ of all smooth Riemannian metrics on $M$ is an open,
positive cone in $\s$.  As such, $\M$ is trivially a Fréchet manifold.
Additionally, its tangent space $T_g \M$ at any point $g \in \M$ can
be canonically identified with $\s$.

For each $x \in \M$, we define $\satx := S^2 T^*_x M$ to be the set of
all symmetric $(0, 2)$-tensors based at $x$.  We denote by $\Matx
\subset \satx$ the open, positive cone of all such tensors that induce
a positive definite scalar product on $T_x M$.  For each $a \in
\Matx$, there is a natural scalar product on $\satx$ given by
\begin{equation}\label{eq:36}
  \langle b, c \rangle_a := \tr_a(bc) = a^{ij} a^{lm} b_{il} c_{jm}
  \quad \textnormal{for all}\ b, c \in \satx.
\end{equation}
The last expression in the above requires the choice of some
coordinates around $x$, but the resulting value will clearly be
coordinate-independent.  Furthermore, $\langle \cdot, \cdot \rangle_a$
is positive definite for each $a \in \Matx$ \cite[Lemma
2.35]{clarked):_compl_of_manif_of_rieman}.  We will denote the norm
associated to (\ref{eq:36}) by $\abs{\, \cdot \,}_a$, i.e.,
\begin{equation}\label{eq:38}
  \abs{b}_a = \sqrt{\langle b, b \rangle_a}.
\end{equation}

By integrating the scalar product (\ref{eq:36}), we can obtain a
scalar product on elements of $\s$, giving us a Riemannian metric on
$\M$.  This is the \emph{$L^2$ metric}, and explicitly, it is given by
the following: For $g \in \M$ and $h, k \in \s \cong T_g \M$,
\begin{equation}\label{eq:37}
  (h, k)_g := \integral{M}{}{\tr_g(h k)}{d \mu_g},
\end{equation}
where $\mu_g$ denotes the volume form induced by $g$.  This is indeed
a Riemannian metric.  Firstly, it is positive definite, as $\langle
\cdot, \cdot \rangle_{g(x)}$ is for each $x \in M$.  And secondly,
$(\cdot, \cdot)$ varies smoothly with $g$, as shown by Ebin
\cite[pp.~18--19]{ebin70:_manif_of_rieman_metric}.  Additionally, this
metric is invariant under the diffeomorphism group $\D$, which acts by
pull-back \cite[\S 6.1.2]{clarked):_compl_of_manif_of_rieman}.  (That
is, $\D$ acts by isometries.)  Throughout the rest of this paper, we
denote the Riemannian distance function of $(\cdot, \cdot)$ by $d$.
We denote the norm on $\s$ induced by (\ref{eq:37}) by $\normdot_g$,
that is,
\begin{equation*}
  \norm{h}_g = \sqrt{(h, h)_g}.
\end{equation*}

The curvature (cf.~\cite[\S 1]{freed89:_basic_geomet_of_manif_of},
\cite[\S 2.5--2.9]{gil-medrano91:_rieman_manif_of_all_rieman_metric})
and geodesics (cf.~\cite[Thm.~2.3]{freed89:_basic_geomet_of_manif_of},
\cite[Thm.\ 3.2]{gil-medrano91:_rieman_manif_of_all_rieman_metric}) of
the $L^2$ metric have been explicitly computed.  We will not need them
here, except for some very special geodesics.  If we let $\pos \subset
C^\infty(M)$ denote the space of smooth, positive functions on $M$,
then $\pos$ acts on $\M$ by pointwise multiplication (conformal
changes), and we have the following result.

\begin{proposition}[{\cite[Prop.~2.1]{freed89:_basic_geomet_of_manif_of}}]\label{prop:8}
  The geodesic starting at $g_0 \in \M$ with initial tangent vector
  $\rho g_0$, where $\rho \in C^\infty(M)$, is given by
  \begin{equation*}
    g_t = \left( 1 + n \frac{t}{4} \rho \right)^{4/n} g_0.
  \end{equation*}
  In particular, the exponential mapping $\exp_{g_0}$ is a
  diffeomorphism from the open set of functions $\{ \rho \in \pos \mid
  \rho > -4/n \}$ onto $\pos \cdot g_0$.
\end{proposition}

We must remark here that the $L^2$ metric is a so-called \emph{weak
  Riemannian metric} (cf.~\cite[\S
3]{clarke:_metric_geomet_of_manif_of_rieman_metric}), which means that
each tangent space $T_g \M$ is incomplete with respect to $(\cdot,
\cdot)_g$, or equivalently that the topology induced by $(\cdot,
\cdot)_g$ on $T_g \M$ is weaker than the manifold topology.  In fact,
as a consequence, standard results in Riemannian geometry---even the
existence of the Levi-Civita connection, curvature tensor, and
geodesics---do not hold \emph{a priori}.  However, Ebin \cite[\S
4]{ebin70:_manif_of_rieman_metric} gave a direct proof that the
Levi-Civita connection of the $L^2$ metric exists, and so in
particular, the curvature tensor and geodesics exist as well.  One can
even show
\cite[Thm.~3.4]{gil-medrano91:_rieman_manif_of_all_rieman_metric} that
the exponential mapping $\exp_g$ at $g \in \M$ is a real-analytic
diffeomorphism between subsets of $T_g \M$ and $\M$ that are open in
the manifold topology.

On the other hand, a serious difficulty in studying the distance
function $d$ is that $\exp_g$ is \emph{not} defined on any subset of
$T_g \M$ that is open with respect to $(\cdot, \cdot)_g$, and its
image does not contain any open $d$-ball around $g$.  In such a
situation, it can happen that the Riemannian distance function is only
a \emph{pseudo}metric, i.e., positive definiteness is not guaranteed.
(See \cite{michor05:_vanis_geodes_distan_spaces_of,michor06:_rieman_geomet_spaces_of_plane_curves} for examples where
the distance function even vanishes everywhere!)  As we showed in
\cite{clarke:_metric_geomet_of_manif_of_rieman_metric}, though, this
is not the case here---$d$ in fact induces a metric space structure on
$\M$.  Nevertheless, there do exist points that are arbitrarily close
with respect to $d$, yet are not connected by a geodesic.  This is
another reason to pursue our alternative way to estimate $d$.

\subsection{The completion of $\M$}\label{sec:completion-m}

\begin{convention}\label{cvt:2}
  For the remainder of the paper, we fix an element $g \in \M$.
  Whenever we refer to the $L^2$ norm $\normdot_g$ and
  $L^2$-convergence, we mean that induced by $g$ unless we explicitly
  state otherwise.  The designation nullset refers to
  Lebesgue-measurable subsets of $M$ that have measure zero with
  respect to the volume form $\mu_g$ of $g$.  If we say that something
  holds almost everywhere, we mean that it holds outside of a
  $\mu_g$-nullset.

  If we have a tensor $h \in \s$, we denote by the capital letter $H$
  the tensor obtained by raising an index with $g$, i.e., locally
  $H^i_j := g^{ik} h_{kj}$.  We sometimes also write $H = g^{-1} h$.
  Given a point $x \in M$ and an element $a \in \Matx$, the capital
  letter $A$ means the same---i.e., we assume some coordinates and
  write $A = g(x)^{-1} a$, though for readability we will generally
  omit $x$ from the notation.
\end{convention}

To give the description of the completion of $\M$ mentioned in the
introduction, we will have to consider generalizations of Riemannian
metrics.  In particular, we must allow degenerations in both
regularity and positive definiteness.  The next definition covers
these objects.

\begin{definition}\label{dfn:5}
  We denote by $S^2 T^* M$ the bundle of symmetric $(0, 2)$ tensors on
  $M$.  A section of $S^2 T^* M$ (a symmetric $(0, 2)$-tensor field)
  that induces a positive semi-definite scalar product on each tangent
  space of $M$ is called a \emph{Riemannian semimetric} (or simply
  \emph{semimetric}).  (Note we do not make any assumptions on the
  regularity of this section.)

  We call a semimetric $\tilde{g}$ \emph{measurable} if it is a
  measurable section of $S^2 T^* M$, and we denote by $\Mm$ the space
  of all measurable semimetrics on $\M$.
\end{definition}

Note that if $\tilde{g} \in \Mm$, then $\mu_{\tilde{g}} := \sqrt{\det
  \tilde{g}} \, dx^1 \wedge \cdots \wedge dx^n$ is a measurable
$n$-form on $M$ with nonnegative coefficient in each coordinate chart.
Thus, it induces a (Lebesgue) measure on $M$ in the usual way.  This
measure is absolutely continuous with respect to our standard measure
$\mu_g$.  In particular, a sequence that converges a.e.~with respect
to $\mu_g$ converges a.e. with respect to $\mu_{\tilde{g}}$.

For any measurable subset $E \subseteq M$, we denote by $\Vol(E,
\tilde{g})$ (sometimes also denoted $\mu_{\tilde{g}}(E)$) the measure
of the subset $E$ with respect to $\mu_{\tilde{g}}$.  Furthermore, we
define
\begin{equation*}
  \Mf :=
  \left\{
    \tilde{g} \in \Mm \midmid \Vol(M, \tilde{g}) < \infty
  \right\},
\end{equation*}
so that $\Mf$ is the space of \emph{finite-volume measurable
  semimetrics}.

Note that if $\nu$ is a measurable, positive $n$-form (meaning with
positive coefficient) and $\mu$ is a measurable, nonnegative $n$-form,
then there exists a unique nonnegative function $(\mu / \nu)$ on $M$
that satisfies
\begin{equation*}
  \mu =
  \bigg(
  \frac{\mu}{\nu}
  \bigg) \nu.
\end{equation*}
In particular, if $\nu(M) < \infty$ (the only case we will be
concerned with here), then $(\mu / \nu)$ is the Radon-Nikodym
derivative of $\mu$ with respect to $\nu$.

The nonnegative $n$-form of a semimetric $\tilde{g}$ vanishes at
exactly those points where $\tilde{g}$ fails to be positive definite.
This motivates the following definition.

\begin{definition}\label{dfn:7}
  Let $\tilde{g} \in \Mm$.  The \emph{deflated set} of $\tilde{g}$ is
  defined by
  \begin{equation*}
    X_{\tilde{g}} := \{ x \in M \mid \det \tilde{G}(x) = 0 \}.
  \end{equation*}

  Analogously, if $\{g_k\} \in \Mm$ is a sequence, then we define
  \begin{equation*}
    D_{\{g_k\}} := \{ x \in M \mid \det G_k(x) \rightarrow 0 \}.
  \end{equation*}
\end{definition}

Note that, clearly, $\Vol(X_{\tilde{g}}, \tilde{g}) = 0$ for all
$\tilde{g} \in \Mm$.

We next define an equivalence relation on $\Mm$ by saying $g_0 \sim
g_1$ if and only if the following two statements hold:
\begin{enumerate}
\item The degenerate sets $X_{g_0}$ and $X_{g_1}$ agree up to a
  nullset; and
\item $g_0(x) = g_1(x)$ for almost every $x \in M \setminus (X_{g_0}
  \cup X_{g_1})$.
\end{enumerate}
In other words, we say $g_0 \sim g_1$ if and only if $g_0(x)$ and
$g_1(x)$ differ only where they are both deflated (up to a nullset).
Denote by $\Mmhat := \Mm / {\sim}$ and $\Mfhat := \Mf / {\sim}$ the
quotients by this equivalence relation.

There is a natural notion of convergence that allows us to give an
element of $\Mfhat$ as the limit of certain sequences in $\M$.  This
is defined as follows.

\begin{definition}\label{dfn:8}
  Let $\{g_k\}$ be a sequence in $\M$, and let $[g_0] \in \Mfhat$.  We
  say that $\{g_k\}$ \emph{$\omega$-converges} to $[g_0]$ if for every
  representative $g_0 \in [g_0]$, the following holds:
  \begin{enumerate}
  \item \label{item:4} $\{g_k\}$ is $d$-Cauchy,
  \item \label{item:14} $X_{g_0}$ and $D_{\{g_k\}}$ differ at most by
    a nullset,
  \item \label{item:15} $g_k(x) \rightarrow g_0(x)$ for a.e.~$x \in M
    \setminus D_{\{g_k\}}$, and
  \item \label{item:7} $\sum_{k=1}^\infty d(g_k, g_{k+1}) < \infty$.
  \end{enumerate}
  We call $[g_0]$ the \emph{$\omega$-limit} of the sequence $\{g_k\}$
  and write $g_k \overarrow{\omega} [g_0]$.

  More generally, if $\{g_k\}$ is a $d$-Cauchy sequence containing a
  subsequence that $\omega$-converges to $[g_0]$, then we say that
  $\{g_k\}$ \emph{$\omega$-subconverges} to $[g_0]$.
\end{definition}

Note that the definition of $\omega$-convergence requires a sequence
to be $d$-Cauchy---this is not guaranteed by conditions
\ref{item:14} and \ref{item:15}.  Note also that conditions
\ref{item:14} and \ref{item:15} are the ones that give the
substantial properties of an $\omega$-convergent sequence.  Condition
\ref{item:7} is merely technical, and can always be achieved by
passing to a subsequence, provided the sequence is $d$-Cauchy.

It is not hard to see \cite[Lemma
4.5]{clarke:_compl_of_manif_of_rieman_metric} that if one
representative $g_0 \in [g_0]$ satisfies the conditions in the above
definition, then all representatives do.  Therefore, for convenience
we will usually just write things like $g_k \xrightarrow{\omega} g_0$,
even when the equivalence class of $g_0$ is meant.

Finally, we recall the basic facts about completions of metric spaces,
as well as fix notation.  As with any metric space, the completion of
$(\M, d)$ is a quotient space of the set of Cauchy sequences in $\M$.
We define a pseudometric, which for simplicity is also denoted by $d$,
on Cauchy sequences in $\M$ by
\begin{equation*}
  d(\{g_k\}, \{\tilde{g}_k\}) = \lim_{k \rightarrow \infty} d(g_k, \tilde{g}_k).
\end{equation*}
That this limit exists is a straightforward argument using the Cauchy
sequence property.  It is not hard to see that this $d$ is only a
pseudometric on the set of Cauchy sequences, so to get a metric space,
we must identify Cauchy sequences with distance zero in this
pseudometric.  Thus, we write $\{g_k\} \sim \{\tilde{g}_k\}$ if and
only if $\lim d(g_k, \tilde{g}_k) = 0$, and define
\begin{equation*}
  \overline{\M} = \grpquot{\{\textnormal{Cauchy sequences}\ \{g_k\} \subset \M \}}{{\sim}}.
\end{equation*}
Note that $\M$ is isometrically embedded in $\overline{\M}$ by mapping
a point $g \in \M$ to the constant sequence $\{g\}$.  Furthermore, if
$\{g_k\}$ is a Cauchy sequence, then any subsequence $\{g_{k_l}\}$ is
equivalent to the original sequence.  Therefore, we may pass to
subsequences as we like and still be talking about the same element of
$\overline{\M}$.

By the results of \cite{clarke:_compl_of_manif_of_rieman_metric}
(Theorems 4.17, 4.27, 4.39, and
5.14, as well as Corollary 4.21), each Cauchy
sequence in $\M$ $\omega$-subconverges to some limit $[g_0] \in
\Mfhat$.  Furthermore, two Cauchy sequences $\omega$-subconverge to
the same limit if and only if they are equivalent (i.e., represent the
same element of $\overline{\M}$).  And finally, for each element
$[g_0] \in \Mfhat$, there exists some Cauchy sequence in $\M$
$\omega$-subconverging to it.  Putting this together, we have the
following theorem.

\begin{theorem}[{\cite[Thm.~5.17]{clarke:_compl_of_manif_of_rieman_metric}}]\label{thm:6}
  There exists a natural bijection $\Omega : \overline{\M} \rightarrow
  \Mfhat$.  The map $\Omega$ assigns to each equivalence class of
  Cauchy sequences in $\M$ the unique element of $\Mfhat$ that each of
  its members $\omega$-subconverge to.
\end{theorem}

One goal of this paper is to replace $\omega$-convergence with a
clearer notion.  A slightly unsatisfactory element of the map $\Omega$
is that it requires passing to subsequences of the original Cauchy
sequence.  This is not much of a problem, since subsequences are
equivalent to the original sequence.  The bigger problem is that
$\omega$-convergence does not identify Cauchy sequences---it assumes a
sequence is Cauchy, and the theorems of
\cite{clarke:_compl_of_manif_of_rieman_metric} essentially tell us
that by passing to a subsequence, we obtain the other conditions of
$\omega$-convergence.  By the end of this paper, however, we will have
a more or less explicit condition for a sequence in $\M$ to be Cauchy,
as well as a more complete understanding of how it converges to a
limit in $\Mfhat$.

In light of the above results, we will denote the metric that $\Omega$
induces on $\Mfhat$ again by $d$, and do the same for the pseudometric
thus induced on $\Mf$.  So if we write $d(g_0, g_1)$ with $g_0, g_1
\in \Mf$, it is understood that this is the same as $d(\{g_0^k\},
\{g_1^k\}) = \lim d(g_0^k, g_1^k)$, where $\{g_0^k\}$ and $\{g_1^k\}$
are sequences $\omega$-converging to $g_0$ and $g_1$, respectively.

\subsection{(Quasi-)Amenable subsets}\label{sec:previous-results}

We will need uniform notions of a Riemannian metric on the base
manifold $M$ being ``not too large'' and ``not too small''.  To do so,
we must first fix a ``good'' coordinate chart on $M$ in which we can
evaluate the coefficients of a metric.

\begin{definition}\label{dfn:4}
  We call a finite atlas of coordinates $\{(U_\alpha, \phi_\alpha)\}$
  for $M$ \emph{amenable} if for each $U_\alpha$, there exist a
  compact set $K_\alpha$ and a different coordinate chart $(V_\alpha,
  \psi_\alpha)$ (which does not necessarily belong to $\{ (U_\alpha,
  \phi_\alpha) \}$) such that
  \begin{equation*}
    U_\alpha \subset K_\alpha \subset V_\alpha \quad \textnormal{and}
    \quad \phi_\alpha = \psi_\alpha | U_\alpha.
  \end{equation*}
\end{definition}

\begin{remark}\label{rmk:1}
  One nice property of an amenable atlas is the following.  Since each
  chart of an amenable atlas is a relatively compact subset of a
  different chart on $M$, we see that given any metric $g \in \M$, the
  coefficients $g_{ij}$ of the metric are bounded functions in each
  chart.
\end{remark}

\begin{convention}
  For the remainder of the paper, we work over a fixed amenable
  coordinate atlas $\{(U_\alpha, \phi_\alpha)\}$ for all computations
  and concepts that require local coordinates.  If we say that a
  statement in local coordinates holds at each $x \in M$, then it is
  understood that the statement should hold in each coordinate chart
  of $\{(U_\alpha, \phi_\alpha)\}$ containing $x$.
\end{convention}

With this convention, we can say what it means for a Riemannian
semimetric to be ``not too large''.

\begin{definition}\label{dfn:6}
  We call a Riemannian semimetric \emph{bounded} if we can find a
  constant $C$ such that for all $x \in M$ and all $1 \leq i, j \leq
  n$, we have $\abs{g_{ij}(x)} \leq C$.
\end{definition}

The next definition picks out subsets of $\M$ whose members are
``uniformly bound\-ed'' away from being too large (and too small).

\begin{definition}\label{dfn:3}
  Let $\lambda^{\tilde{G}}_{\min}$ denote the minimal eigenvalue of
  $\tilde{G} = g^{-1} \tilde{g}$.  A subset $\U \subset \M$ is called
  \emph{amenable} if it is of the form
  \begin{equation}\label{eq:16}
    \U = \{ \tilde{g} \in \M \mid \lambda^{\tilde{G}}_{\min} \geq \zeta\
    \textnormal{and}\ |\tilde{g}_{ij}(x)| \leq C\ \textnormal{for all $\tilde{g}
      \in \U$, $x \in M$ and $1 \leq i,j \leq n$} \}
  \end{equation}
  for some constants $C, \zeta > 0$.
  
  We call a subset $\U \subset \M$ \emph{quasi-amenable} if it is of
  the form
  \begin{equation}\label{eq:134}
    \U = \{ \tilde{g} \in \M \mid |\tilde{g}_{ij}(x)| \leq C\ \textnormal{for all $\tilde{g}
      \in \U$, $x \in M$ and $1 \leq i,j \leq n$} \}
  \end{equation}
  for some constant $C \geq 0$.

  For such a subset, we denote by $\U^0$ its closure with respect to
  the $L^2$ norm $\normdot_g$.
\end{definition}

Note also that (quasi-)amenable subsets are convex because of the
convexity of the absolute value and the concavity of the minimal
eigenvalue.

\begin{remark}\label{rmk:2}
  A couple of remarks about these subsets are in order:

  \begin{enumerate}
  \item Definition \ref{dfn:3} differs from the way that
    (quasi-)amenable subsets were defined in
    \cite{clarke:_compl_of_manif_of_rieman_metric}, in that here we
    define them to be maximal with respect to the bounds in
    (\ref{eq:16}) and (\ref{eq:134}).  This will turn out to be the
    most convenient definition.
  \item We could have defined (quasi-)amenable subsets in a
    coordinate-independent way by replacing the condition
    $\abs{\tilde{g}_{ij}} \leq C$ with the condition that for each $x
    \in M$, the set $\absgx{\tilde{g}(x)} \leq C'$ for some other
    constant $C'$.  This would have been completely equivalent for all
    intents and purposes, and is more satisfactory in that it does not
    depend on a choice of coordinate atlas.  However, it would have
    caused the inconvenience of being incompatible with the
    definitions and results of
    \cite{clarke:_compl_of_manif_of_rieman_metric}, at least without a
    good deal of additional remarks at points where we use those
    results.
  \end{enumerate}
\end{remark}

We will not need amenable subsets much in this paper, though they will
come up for technical reasons at one point soon.  The main point of
introducing quasi-amenable subsets is that within a quasi-amenable
subset, we can control the $d$-distance between two metrics using
their distance in the fixed $L^2$ norm $\normdot_g$.  Indeed, we have
the following results, which certainly do not hold on all of $\M$.

\begin{theorem}[{\cite[Thm.~5.12]{clarke:_compl_of_manif_of_rieman_metric}}]\label{thm:7}
  Let $\U \subset \M$ be quasi-amenable.  Then for all $\epsilon > 0$,
  there exists $\delta > 0$ such that if $g_0, g_1 \in \cl(\U)$ (where
  $\cl(\U)$ denotes the topological closure of $\U \subset \s$) with
  $\norm{g_0 - g_1}_g < \delta$, then $d(g_0, g_1) < \epsilon$.
\end{theorem}

\begin{proposition}[{\cite[Prop.~5.13]{clarke:_compl_of_manif_of_rieman_metric}}]\label{prop:5}
  Suppose $g_0 \in \U^0$ for some quasi-amenable subset $\U \subset
  \M$.  Then for any sequence $\{g_k\}$ in $\U$ that $L^2$-converges
  to $g_0$, $\{g_k\}$ is $d$-Cauchy and there exists a subsequence
  $\{g_{k_l}\}$ that $\omega$-converges to $g_0$.
\end{proposition}

Note that by the discussion of the completion of $\M$ in Section
\ref{sec:completion-m}, the Cauchy sequence $\{g_k\}$ in the above
proposition is equivalent to its $\omega$-convergent subsequence
$\{g_{k_l}\}$.  Thus, we have $\lim d(g_k, g_0) = \lim d(g_{k_l}, g_0)
= 0$, and we get the following corollary.

\begin{corollary}\label{cor:2}
  Suppose $g_0 \in \U^0$ for some quasi-amenable subset $\U \subset
  \M$.  Then for any sequence $\{g_k\}$ in $\U$ that $L^2$-converges
  to $g_0$, we have $g_k \xrightarrow{d} g_0$.
\end{corollary}

\subsection{Properties of the metric $d$}\label{sec:prop-omega-conv}

We now turn to a review of results on the behavior of $d$ that were
established in \cite{clarke:_compl_of_manif_of_rieman_metric} and
\cite{clarke:_metric_geomet_of_manif_of_rieman_metric}.  We will also
need to extend some of these results to more general settings.

One extremely important aspect of the topology induced by $d$ on $\M$
is the fact that the volumes of measurable subsets behave
continuously, as the next theorem shows.

\begin{theorem}[{\cite[Thm.~4.20]{clarke:_compl_of_manif_of_rieman_metric}}]\label{thm:9}
  Let $\{g_k\}$ $\omega$-converge to $g_0 \in \Mf$, and let $Y
  \subseteq M$ be any measurable subset.  Then $\Vol(Y, g_k)
  \rightarrow \Vol(Y, g_0)$.
\end{theorem}

Indeed, we have the following theorem, which extends \cite[Lemma
12]{clarke:_metric_geomet_of_manif_of_rieman_metric}.

\begin{lemma}\label{lem:3}
  Let $g_0, g_1 \in \Mf$.  Then for any measurable subset $Y \subseteq
  M$,
  \begin{equation*}
    \left| \sqrt{\Vol(Y,g_1)} - \sqrt{\Vol(Y,g_0)} \right| \leq \frac{\sqrt{n}}{4} d(g_0,g_1).
  \end{equation*}
\end{lemma}
\begin{proof}
  Let $\{g_0^k\}$ and $\{g_1^k\}$ be any sequences in $\M$ that
  $\omega$-converge to $g_0$ and $g_1$, respectively.  By \cite[Lemma
  12]{clarke:_metric_geomet_of_manif_of_rieman_metric}
  and Theorem \ref{thm:9}, we have
  \begin{align*}
    \left| \sqrt{\Vol(Y,g_1)} - \sqrt{\Vol(Y,g_0)} \right| &= \lim_{k
      \rightarrow \infty} \left| \sqrt{\Vol(Y,g_1^k)} - \sqrt{\Vol(Y,g_0^k)} \right| \\
    &\leq \lim_{k \rightarrow \infty} \frac{\sqrt{n}}{4}
    d(g_0^k,g_1^k) \\
    &= \frac{\sqrt{n}}{4} d(g_0,g_1).
  \end{align*}
\end{proof}

By the last theorem, the difference in the volumes of a given subset
bounds the distance in $d$ from below.  Surprisingly, we also have the
following result, which bounds the distance between two metrics based
on the volume of the subset on which they differ.

\begin{theorem}[{\cite[Thm.~4.34]{clarke:_compl_of_manif_of_rieman_metric}}]\label{thm:8}
  Let $\U$ be any amenable subset with $L^2$-closure $\U^0$.  Suppose
  that $g_0, g_1 \in \U^0$, and let $E := \carr (g_1 - g_0) = \{ x \in
  M \mid g_0(x) \neq g_1(x) \}$.  Then there exists a constant $C(n)$
  depending only on $n = \dim M$ such that
  \begin{equation*}
    d (g_0, g_1) \leq C(n) \left( \sqrt{\Vol(E, g_0)} +
      \sqrt{\Vol(E,g_1)} \right).
  \end{equation*}
  In particular, $C(n)$ does not depend on $g_0$, $g_1$, or $\U$, and
  we have
  \begin{equation*}
    \diam \left( \{ \tilde{g} \in \U^0 \mid \Vol(M, \tilde{g}) \leq
      K \} \right) \leq 2 C(n) \sqrt{K}.
  \end{equation*}
\end{theorem}

Thus, the metrics $g_0$ and $g_1$ can differ \emph{arbitrarily} on the
subset $E$, and still their distance from one another will be
uniformly bounded by the intrinsic volume of $E$.  One consequence of
this is that metrics with very small volume are close with respect to
$d$, despite the fact that they may be geometrically very different.
For example, a torus with latitudinal radius large and longitudinal
radius small---a wide, thin torus, geometrically almost a circle---has
small distance from a torus with both radii small---geometrically
almost a point.

Another consequence of Theorem \ref{thm:8}, which we mention in
passing purely for its intrinsic interest, concerns certain well-known
subspaces of metrics in $\M$.

\begin{corollary}\label{cor:3}
  With respect to $d$, the following submanifolds of $\M$ lie within a
  bounded region:
  \begin{enumerate}
  \item\label{item:9} For $\mu$ a smooth volume form on $M$, the
    submanifold $\Mmu$ of metrics inducing the volume form $\mu$.
  \item\label{item:10} For $\tilde{g} \in \M$ any metric, the orbit of
    $\tilde{g}$ under the action (by pull-back) of the diffeomorphism
    group of $M$.
  \item\label{item:11} For $\lambda > 0$ any number, the submanifold
    $\M_\lambda$ of metrics having total volume $\lambda$.
  \item\label{item:12} For $\lambda > 0$, the submanifold
    $\M_\lambda^0$ of metrics having total volume less than $\lambda$.
  \item\label{item:13} If the base manifold $M$ is a surface of genus
    $p \geq 2$, the submanifold $\Mhyp$ of hyperbolic metrics on $M$
    (having constant Gaussian curvature $-1$).
  \end{enumerate}

  Furthermore, since $\M \cong \M_\lambda \times \R_{> 0}$, we have
  that $\M$ is diffeomorphic to the product of a $d$-bounded subset
  with $\R_{> 0}$.
\end{corollary}
\begin{proof}
  The manifolds \eqref{item:9}--\eqref{item:11} clearly consist of
  metrics all having the same volume, and \eqref{item:12} consists of
  metrics with volume bounded above by $\lambda$.  Furthermore, by the
  Gauß--Bonnet formula, a hyperbolic metric on a surface has total
  volume equal to $4 \pi (p - 1)$.  Therefore \ref{item:13} also
  consists of metrics all having the same volume.  So the result is
  implied by Theorem \ref{thm:8}.
\end{proof}

With this digression into curiosities out of the way, we return to
establishing the results we need later.  The next proposition extends
Theorem \ref{thm:8} to the entire completion of $\M$.

In the proof of this proposition, and for the remainder of the paper,
we denote the characteristic function of any set $E \subseteq M$ by
$\chi(E)$.

\begin{proposition}\label{prop:1}
  Let $g_0, g_1 \in \Mf$ and $A := \carr(g_1 - g_0)$.  Then
  \begin{equation*}
    d(g_0, g_1) \leq C(n)
    \left(
      \sqrt{\Vol(A, g_0)} + \sqrt{\Vol(A, g_1)}
    \right),
  \end{equation*}
  where $C(n)$ is the same constant as in Theorem \ref{thm:8}.
\end{proposition}
\begin{proof}
  First, for $\alpha = 0, 1$, define
  \begin{gather*}
    E_\alpha^k := \left\{ x \midmid \lambda^{G_\alpha}_{\min} \geq
      \frac{1}{k}\ \textnormal{and}\ \abs{(g_\alpha)_{ij}(x)} \leq k;\
      \alpha
      = 0,1 \right\}, \\
    g^k_\alpha := \chi(E_\alpha^k) g_\alpha+ \chi(M \setminus
    E_\alpha^k) g.
  \end{gather*}
  Then there exists an amenable subset $\U_k$ such that $\U_k^0$
  contains both $g_0^k$ and $g_1^k$.  Furthermore,
  \begin{equation*}
    \carr(g_1^k - g_0^k) \subseteq A \cup
    \left(
      \left(
        E_0^k \cup E_1^k
      \right) \setminus
      \left(
        E_0^k \cap E_1^k
      \right)
    \right)
  \end{equation*}
  since we have only modified $g_\alpha$ on $E_\alpha^k$, and on
  $E_0^k \cap E_1^k$, $g_0^k = g = g_1^k$.  But we also have that if
  $x \notin A$, then $g_0(x) = g_1(x)$, so in this case $x \in E_0^k$
  if and only if $x \in E_1^k$.  In other words, $x \notin A$ implies
  that either $x \in E_0^k \cap E_1^k$ or $x \notin E_0^k \cup E_1^k$.
  From this we see that $( E_0^k \cup E_1^k ) \setminus ( E_0^k \cap
  E_1^k ) \subseteq A$, implying that $\carr(g_1^k - g_0^k) \subseteq
  A$.

  Thus, by Theorem \ref{thm:8}, we have
  \begin{equation}\label{eq:2}
    d(g_0^k, g_1^k) \leq C(n)
    \left(
      \sqrt{\Vol(A, g_0^k)} + \sqrt{\Vol(A, g_1^k)}
    \right).
  \end{equation}
  If we can now show that $g^k_\alpha \xrightarrow{\omega} g_\alpha$
  for $\alpha = 0, 1$, then $\Vol(A, g^k_\alpha) \rightarrow \Vol(A,
  g_\alpha)$ by Theorem \ref{thm:9}.  This, together with
  (\ref{eq:2}), would give the result by taking the limit of both
  sides of the inequality.

  Since $\chi(E_\alpha^k)$ converges a.e.~to $\chi(M \setminus
  X_{g_\alpha})$ as $k \rightarrow \infty$ (recall that $X_{g_\alpha}$
  is the deflated set of $g_\alpha$), all the conditions for
  $g^k_\alpha$ to $\omega$-converge to $g_\alpha$ are clear, except
  that we must verify that $\{g^k_\alpha\}$ is a $d$-Cauchy sequence.

  Now, if $k, l \in \N$, we easily see that $E_\alpha^k \subseteq
  E_\alpha^{k+l}$, and that $g^k_\alpha$ and $g^{k+l}_\alpha$ only
  differ on $E_\alpha^{k+l} \setminus E_\alpha^k$.  Furthermore, the
  amenable subset $\U_{k+l}$ can clearly be chosen such that
  $g^k_\alpha \in \U_{k+l}$.  Thus, using Theorem \ref{thm:8} again,
  we see
  \begin{equation*}
    d(g^k_\alpha, g^{k+l}_\alpha) \leq C(n)
    \left(
      \sqrt{\Vol(E_\alpha^{k+l} \setminus E_\alpha^k, g^k_\alpha)} + \sqrt{\Vol(E_\alpha^{k+l}
        \setminus E_\alpha^k, g^{k+l}_\alpha)}
    \right).
  \end{equation*}
  Since on $E_\alpha^{k+l} \setminus E_\alpha^k$, we have $g^k_\alpha
  = g$ and $g^{k+l}_\alpha = g_\alpha$, we can rewrite the above
  inequality as
  \begin{equation}\label{eq:5}
    d(g^k_\alpha, g^{k+l}_\alpha) \leq C(n)
    \left(
      \sqrt{\Vol(E_\alpha^{k+l} \setminus E_\alpha^k, g)} + \sqrt{\Vol(E_\alpha^{k+l}
        \setminus E_\alpha^k, g_\alpha)}
    \right).
  \end{equation}

  Next, note that by using the fact that $E_\alpha^{k+l} \subseteq M
  \setminus X_{g_\alpha}$, we can estimate
  \begin{equation}\label{eq:6}
    \Vol(E_\alpha^{k+l} \setminus E_\alpha^k, g_\alpha) \leq \Vol((M
    \setminus X_{g_\alpha}) \setminus E_\alpha^k,
    g_\alpha)
  \end{equation}
  and
  \begin{equation}\label{eq:40}
    \Vol(E_\alpha^{k+l} \setminus E_\alpha^k, g) \leq \Vol((M
    \setminus X_{g_\alpha}) \setminus E_\alpha^k,
    g).
  \end{equation}
  On the other hand, as we already noted, $\chi(E_\alpha^k)$ converges
  a.e.~to $\chi(M \setminus X_{g_\alpha})$ as $k \rightarrow \infty$,
  and so $\chi((M \setminus X_{g_\alpha}) \setminus E_\alpha^k)
  \rightarrow 0$ a.e.  Since we also have that $\chi((M \setminus
  X_{g_\alpha}) \setminus E_\alpha^k) \leq 1$ for each $k \in \N$, and
  the constant function $1$ is $\mu_{g_\alpha}$-integrable since
  $g_\alpha$ has finite volume, we can apply the Lebesgue Dominated
  Convergence Theorem to see that
  \begin{equation}\label{eq:7}
    \lim_{k \rightarrow \infty} \Vol((M
    \setminus X_{g_\alpha}) \setminus E_\alpha^k, g_\alpha) =
    \lim_{k \rightarrow \infty} \integral{M}{}{\chi(M \setminus
      E_\alpha^k)}{d \mu_{g_\alpha}} = 0.
  \end{equation}
  Analogously, $\lim_{k \rightarrow \infty} \Vol((M \setminus
  X_{g_\alpha}) \setminus E_\alpha^k, g) = 0$.

  Combining this with (\ref{eq:5}), (\ref{eq:6}), \eqref{eq:40} and
  (\ref{eq:7}) shows that for $k$ large enough, $d(g_k^\alpha,
  g_{k+l}^\alpha)$ becomes arbitrarily small---independently of
  $l$---from which it follows that $\{g_k^\alpha\}$ is Cauchy, as was
  to be shown.
\end{proof}

We now have most of the prerequisite facts necessary for our study of
the topology induced by $d$.  However, we will still need to review
other convergence notions, including the metric $\TM$ mentioned in the
introduction.

\subsection{Convergence in measure}\label{sec:convergence-measure}

In the theory of $L^p$ spaces, convergence of functions in measure
plays an important role.  We will need a straightforward
generalization of this here, where we just replace the usual absolute
value on $\R$ with the norm $\absgx{\, \cdot \,}$.

\begin{definition}\label{dfn:1}
  $\{g_k\} \subset \Mm$ converges in ($\mu$-)measure to $g_0 \in \Mm$,
  where $\mu$ is some Lebesgue measure on $M$, if for all $\epsilon >
  0$,
  \begin{equation*}
    \lim_{k \rightarrow \infty} \mu
    \left(
      \left\{
        x \in M \midmid \absgx{g_0(x) - g_k(x)}
        \geq \epsilon
      \right\}
    \right) = 0.
  \end{equation*}
  If $\mu$ is omitted, it is assumed that $\mu = \mu_g$, the volume
  form of our fixed reference metric $g$.
\end{definition}

The following lemma will allow us to translate convergence in
$\mu_g$-measure to convergence in other measures.

\begin{lemma}\label{lem:14}
  Let $\mu$ be a finite Lebesgue measure on $M$, and let $\nu$ be
  another finite Lebesgue measure on $M$ that is absolutely continuous
  with respect to $\mu$.  If the sequence $\{g_k\} \subset \Mm$
  converges to $g_0 \in \Mm$ in $\mu$-measure, then the sequence
  converges to $g_0$ in $\nu$-measure as well.
\end{lemma}
\begin{proof}
  The lemma follows easily if one observes that for each $\epsilon >
  0$, there exists $\delta$ such that for all measurable $E \subseteq
  M$, $\mu(E) < \delta$ implies $\nu(E) < \epsilon$.  This, in turn,
  follows from the fact that
  \begin{equation*}
    \lim_{C \rightarrow \infty} \mu
    \left(
      \left\{
        x \in M \midmid \frac{d \nu}{d \mu}(x) \geq C
      \right\}
    \right) = 0,
  \end{equation*}
  where $d \nu / d \mu$ denotes the Radon-Nikodym derivative.
\end{proof}

In particular, the previous lemma applies to the case when $\mu =
\mu_g$ and $\nu = \mu_{\tilde{g}}$, where $\tilde{g} \in \Mf$.

Finally, we need a quick definition that gives a strong type of
convergence of measures that will come up later.

\begin{definition}\label{dfn:2}
  Let $\mu_k$ and $\mu$ be nonnegative Lebesgue measures on $M$.  We
  say that $\{\mu_k\}$ \emph{converges uniformly to $\mu$} iff for all
  $\epsilon > 0$, there exists $k_0 \in \N$ such that for all $k \geq
  k_0$ and for all $E \subseteq M$ measurable, $\abs{\mu(E) -
    \mu_k(E)} < \epsilon$.
\end{definition}

\begin{remark}
  Notice that this convergence is stronger than, for example, weak-$*$
  convergence (sometimes also just called weak convergence) of
  measures.  It gives a topology on the space of measures on $M$ that
  is equivalent to the topology induced by the supremum norm on the
  vector space of signed measures on $M$, where we set $\abs{\mu} :=
  \sup \abs{\mu(E)}$, with the supremum ranging over all measurable
  subsets $E \subseteq M$.
\end{remark}

The usefulness of this definition is given by its connection to the
Radon-Nikodym derivative.

\begin{lemma}\label{lem:18}
  Let $\mu_k$ and $\mu$ be nonnegative Lebesgue measures on $M$, and
  let $\nu$ be any Lebesgue measure with respect to which $\mu$ and
  all $\mu_k$ are absolutely continuous.  Furthermore, assume that
  $\mu_k(M), \mu(M), \nu(M) < \infty$.  Then uniform convergence of
  $\{\mu_k\}$ to $\mu$ is equivalent to $L^1$-convergence of the
  Radon-Nikodym derivatives:
  \begin{equation*}
    \frac{d \mu_k}{d \nu} \xrightarrow{L^1(M, \nu)} \frac{d \mu}{d \nu}.
  \end{equation*}
\end{lemma}

In order to prove this lemma, we need a characterization of
convergence of $L^p$ functions.  Let $(\Sigma, \nu)$ be a measure
space, and recall that a collection of measurable functions
$\mathcal{G}$ on $\Sigma$ is called \emph{uniformly absolutely
  continuous} if the following holds: For all $\epsilon > 0$, there
exists $\delta > 0$ such that if $E \subseteq \Sigma$ is measurable
with $\nu(E) < \delta$, then
\begin{equation*}
  \integral{E}{}{\abs{f}}{d \nu} < \epsilon \quad \textnormal{for all}\ f \in \mathcal{G}.
\end{equation*}
It is not hard to see that if $\nu(\Sigma) < \infty$, then any finite
set of functions is uniformly absolutely continuous.

With this definition, we have the following result.

\begin{theorem}[{\cite[Thm.~8.5.14]{rana02:_introd_to_measur_and_integ}}]\label{thm:11}
  Let $(X, \nu)$ be a measure space with $\nu(X) < \infty$, and let
  $f$ be a measurable function on $X$.  Furthermore, let $f_k$ be a
  sequence of functions in $L^p(X, \nu)$.  Then the following
  statements are equivalent.
  \begin{enumerate}
  \item $f_k \rightarrow f$ in $L^p(X, \nu)$.
  \item $\{ |f_k|^p \mid k \in \N \}$ is uniformly absolutely
    continuous and $f_k \rightarrow f$ in measure.
  \end{enumerate}
\end{theorem}

We can now use this to prove the lemma.

\begin{proof}[Proof of Lemma \ref{lem:18}]
  That $L^1$-convergence implies uniform convergence is a
  straightforward argument, so we turn to the proof of the converse
  statement.
  
  By Theorem \ref{thm:11}, it suffices to show that the set of
  functions $\{ d \mu_k / d \nu \mid k \in \N \}$ is uniformly
  absolutely continuous (with respect to $\nu$), and that $d \mu_k / d
  \nu$ converges to $d \mu / d \nu$ in $\nu$-measure.  (That $d \mu /
  d \nu$ and each $d \mu_k / d \nu$ are $L^1$ functions is implied by
  $\mu_k(M), \mu(M) < \infty$.)

  To show that $\{ d \mu_k / d \nu \mid k \in \N \}$ is uniformly
  absolutely continuous, let $\epsilon > 0$ be given.  Since without
  loss of generality, we can forget a finite number of functions from
  the set, we may restrict to $k$ large enough that $|\mu(E) -
  \mu_k(E)| < \epsilon / 2$ for all measurable $E \subseteq M$.  Since
  $\mu$ is absolutely continuous with respect to $\nu$, there exists
  $\delta > 0$ such that if $E$ is measurable and $\nu(E) < \delta$,
  then $\mu(E) < \epsilon / 2$.  So, let $Y \subseteq M$ be any
  measurable subset with $\nu(Y) < \delta$.  Then
  \begin{equation*}
    \integral{Y}{}{\abs{ \frac{d \mu_k}{d \nu} }}{d
      \nu} = \mu_k(Y) < \mu(Y) + \epsilon
    / 2 < \epsilon,
  \end{equation*}
  showing that $\{ d \mu_k / d \nu \mid k \in \N \}$ is uniformly
  absolutely continuous with respect to $\nu$.

  To see that $d \mu_k / d \nu$ converges to $d \mu / d \nu$ in
  $\nu$-measure, assume the contrary.  Thus, there exists $\epsilon >
  0$ such that if
  \begin{equation*}
    E_k^{\epsilon+} := \left\{
      x \in M \midmid 
      \frac{d \mu}{d \nu}
      - 
      \frac{d \mu_k}{d \nu}
      \geq \epsilon
    \right\} \qquad \textnormal{and} \qquad E_k^{\epsilon-} := \left\{
      x \in M \midmid 
      \frac{d \mu_k}{d \nu}
      - 
      \frac{d \mu}{d \nu}
      \geq \epsilon
    \right\},
  \end{equation*}
  then
  \begin{equation*}
    \limsup_{k \rightarrow \infty}
    \nu (E_k^{\epsilon+} \cup E_k^{\epsilon-}) = \delta > 0.
  \end{equation*}
  By additivity of $\nu$, either $\limsup \nu(E_k^{\epsilon+}) \geq
  \delta / 2$, or $\limsup \nu(E_k^{\epsilon-}) \geq \delta / 2$.
  Without loss of generality, say that the former holds.  This then
  gives that for all $k_0 \in \N$, there exists $k \geq k_0$ such that
  \begin{equation*}
    \mu(E_k^{\epsilon+}) - \mu_k(E_k^{\epsilon+}) = \integral{E_k^{\epsilon+}}{}{
      \left[
        \frac{d \mu}{d \nu} - \frac{d \mu_k}{d \nu}
      \right]}{d \nu} \geq \epsilon \cdot \frac{\delta}{2} > 0.
  \end{equation*}
  This, however, is in direct contradiction of the assumption that
  $\mu_k$ converges uniformly to $\mu$.
\end{proof}

With these preliminaries on convergence in measure spaces out of the
way, we now turn to our detailed discussion of the metric structure
induced on $\M$ by $\TM$.

\section{The metric $\TM$}\label{sec:metric-tm}

\subsection{Motivation and definition}\label{sec:motiv-defin}

As mentioned in the Introduction, computing $d$ for arbitrary points
$g_0, g_1 \in \M$ involves an infinite-dimensional problem, since we
have to find the infimum of the expression
\begin{equation}\label{eq:31}
  \begin{aligned}
    L(g_t) &= \integral{0}{1}{\norm{g'_t}_{g_t}}{dt} =
    \integral{0}{1}{ \left( \integral{M}{}{\tr_{g_t}((g'_t)^2)}{d
          \mu_{g_t}}
      \right)^{1/2}}{dt} \\
    &= \integral{0}{1}{ \left( \integral{M}{}{\tr_{g_t}((g'_t)^2)
          \sqrt{\det G_t}}{d \mu_g} \right)^{1/2}}{dt}
  \end{aligned}
\end{equation}
over all paths $g_t$ connecting $g_0$ and $g_1$.  Furthermore, as
noted at the end of Section \ref{sec:manifold-metrics}, we cannot
reduce the question to one of geodesics even for close-together
points.

One solution, as hinted at in the Introduction, is changing the order
of integration in (\ref{eq:31}).  We have already taken the first step
in this by removing the $t$-dependence from the volume form above.
The second step requires introducing a new Riemannian metric on
$\Matx$.

\begin{definition}\label{dfn:9}
  For each $x \in M$, define a Riemannian metric $\langle \cdot ,
  \cdot \rangle^0$ on $\Matx$ by
  \begin{equation*}
    \langle b , c \rangle^0_{a} = \tr_{a} (b c) \det
    A \quad \textnormal{for all}\ b, c \in T_a \Matx \cong
    \satx.
  \end{equation*}
  (Recall that $A$ denotes $g(x)^{-1} a$, cf.~Convention \ref{cvt:2}.)
  We denote by $\theta^g_x$ the Riemannian distance function of
  $\langle \cdot , \cdot \rangle^0$.
  
  For any measurable $Y \subseteq M$, define a function $\Theta_Y : \M
  \times \M \rightarrow \R$ by
  \begin{equation*}
    \Theta_Y(g_0, g_1) = \integral{Y}{}{\theta^g_x(g_0(x), g_1(x))}{d \mu_g}.
  \end{equation*}
\end{definition}

Thus, determining $\Theta_Y(g_0, g_1)$ indeed involves finding the
distance between $g_0(x)$ and $g_1(x)$ in $\Matx$ and then integrating
this over $M$, as desired.  Note that as $\tgx$ is a Riemannian
distance function on a finite-dimensional manifold, it is as usual a
metric (in particular, it is positive definite).

By the results of \cite[\S
4]{clarke:_metric_geomet_of_manif_of_rieman_metric},
we have that $\Theta_Y$ does not depend on the choice of reference
metric $g$.  Furthermore, $\Theta_Y$ is a pseudometric on $\M$, and in
the special case $Y = M$---the one we will be most concerned
with---$\TM$ is a metric (in the sense of metric spaces).  Using a
Hölder's Inequality argument to get rid of the square root in
(\ref{eq:31}), we also obtained the following relation between $\TM$
and $d$
\cite[Prop.~27]{clarke:_metric_geomet_of_manif_of_rieman_metric}:
\begin{equation}\label{eq:32}
  \TM(g_0, g_1) \leq d(g_0, g_1) \left( \sqrt{n}\, d(g_0, g_1) +
    2 \sqrt{\Vol(M, g_0)} \right)
\end{equation}
for all $g_0, g_1 \in \M$.  This inequality shows that the topology of
$\TM$ is no stronger that that of $d$, in the sense that $d(g_k, g_0)
\rightarrow 0$ implies $\TM(g_k, g_0) \rightarrow 0$ for all $g_0 \in
\M$ and all sequences $\{g_k\} \subset \M$.  If we had such an
inequality with the roles of $\TM$ and $d$ reversed, then we would
have achieved our goal of changing the order of integration in
(\ref{eq:31}), at least as far as topological questions are concerned.
Unfortunately, we do not have such an estimate, but the rest of the
paper is essentially about giving us the topological results we want
in a different way.

\subsection{Fundamental results on $\TM$}\label{sec:fund-results-tmk}

In the rest of this section, we give an investigation of the
properties of $\TM$, starting with a review of results we have already
established elsewhere.  The first such result gives an explicit
description of the completion of $(\Matx, \theta^g_x)$, and is the
first step towards understanding the completion of $(\M, \TM)$.

\begin{theorem}[{\cite[Thm.~4.14]{clarke:_compl_of_manif_of_rieman_metric}}]\label{thm:10}
  For any given $x \in M$, let $\textnormal{cl}(\Matx)$ denote the
  closure of $\Matx \subset \satx$ with regard to the natural
  topology.  Then $\cl(\Matx)$ consists of all positive semidefinite
  $(0,2)$-tensors at $x$.  Let us denote the boundary of $\Matx$, as a
  subspace of $\satx$, by $\partial \Matx$.  We denote the quotient of
  this closure by its boundary by $\overline{\Matx} := \cl(\Matx)
  / \partial \Matx$.

  Then the completion of $(\Matx, \theta^g_x)$ can be identified with
  $\overline{\Matx}$.  The distance function is given by
  \begin{equation}\label{eq:33}
    \theta^g_x([g_0], [g_1]) = \lim_{k \rightarrow
      \infty}\theta^g_x(g^0_k, g^1_k),
  \end{equation}
  where $\{g^0_k\}$ and $\{g^1_k\}$ are any sequences in $\Matx$
  converging (in the topology of $\satx$) to $g_0$ and $g_1$,
  respectively, in $\cl(\Matx)$.
\end{theorem}

From now on, we will drop the equivalence class notation when writing
the $\tgx$-distance between elements of $\overline{\Matx}$, with the
understanding that the formula (\ref{eq:33}) is implied.

Using the above theorem, we can give meaning to $\Theta_Y$ on
$\overline{(\M, d)}$, and even extend the estimate (\ref{eq:32}) to
this space.

\begin{proposition}[{\cite[Prop.~4.25]{clarke:_compl_of_manif_of_rieman_metric}}]\label{prop:6}
  Let $Y \subseteq M$ be measurable.  Then the pseudometric $\Theta_Y$
  on $\M$ can be extended to a pseudometric on $\overline{(\M, d)}
  \cong \Mfhat$ via
  \begin{equation}\label{eq:30}
    \Theta_Y(\{g^0_k\}, \{g^1_k\}) := \lim_{k \rightarrow \infty}
    \Theta_Y(g^0_k, g^1_k)
  \end{equation}
  for any Cauchy sequences $\{g^0_k\}$ and $\{g^1_k\}$.  This
  pseudometric is no stronger than $d$ in the sense that $d(\{g^0_k\},
  \{g^1_k\}) = 0$ implies $\Theta_Y(\{g^0_k\}, \{g^1_k\}) = 0$.  More
  precisely, we have
  \begin{equation*}
    \Theta_Y (\{g^0_k\}, \{g^1_k\}) \leq d(\{g^0_k\}, \{g^1_k\}) \left( \sqrt{n}\, d(\{g^0_k\},
      \{g^1_k\}) + 2 \sqrt{\Vol(M, g_0)} \right),
  \end{equation*}
  where $g_0$ is any element of $\Mf$ that $\{g^0_k\}$
  $\omega$-subconverges to.

  Furthermore, if $\{g^0_k\}$ and $\{g^1_k\}$ are sequences in $\M$
  that $\omega$-converge to $g_0$ and $g_1$, respectively, then the
  formula
  \begin{equation}\label{eq:92}
    \Theta_Y(\{g^0_k\}, \{g^1_k\}) = \integral{Y}{}{\theta_x^g(g_0(x), g_1(x))}{\mu_g(x)}
  \end{equation}
  holds for all $g_0, g_1 \in \M$.
\end{proposition}

In view of the formula (\ref{eq:92}), we will from now on write simply
$\Theta_Y(g_0, g_1)$ for any $g_0, g_1 \in \Mf$, where it is
understood that this quantity is given by (\ref{eq:30}) or,
equivalently, (\ref{eq:92}).

The next result we will make use of gives the pointwise version of
Lemma \ref{lem:3}.

\begin{lemma}[{\cite[Lemma 4.10]{clarke:_compl_of_manif_of_rieman_metric}}]\label{lem:1}
  Let $a_0, a_1 \in \Matx$.  Then
  \begin{equation*}
    \left|
      \sqrt{\det A_1} - \sqrt{\det A_0}
    \right| \leq \frac{\sqrt{n}}{2} \theta^g_x (a_0, a_1).
  \end{equation*}
\end{lemma}

It will be necessary for us to make a straightforward extension of
this result to $\cl(\Matx)$ using (\ref{eq:33}).

\begin{lemma}\label{lem:16}
  Let $a_0, a_1 \in \cl(\Matx)$.  Then
  \begin{equation}\label{eq:26}
    \left|
      \sqrt{\det A_1} - \sqrt{\det A_0}
    \right| \leq \frac{\sqrt{n}}{2} \theta^g_x (a_0, a_1).
  \end{equation}
\end{lemma}
\begin{proof}
  Because of Lemma \ref{lem:1}, it only remains to deal with the case
  that at least one of $a_0$ or $a_1$ belongs to $\partial \Matx$.

  If both belong to $\partial \Matx$, then both sides of \eqref{eq:26}
  are zero, so there is nothing to prove.
  
  We are left with the case that only one belongs to $\partial \Matx$
  (let's say, without loss of generality, that it's $a_0$).  Let
  $\{a_0^k\}$ be a sequence in $\Matx$ that $\tgx$-converges to $a_0$.
  Then $\det A_0^k \rightarrow \det A_0 = 0$ by Theorem \ref{thm:10}, and we also
  have
  \begin{equation*}
    \left|
      \sqrt{\det A_1} - \sqrt{\det A_0^k}
    \right| \leq \frac{\sqrt{n}}{2} \theta^g_x (a_0^k, a_1).
  \end{equation*}
  Taking the limit as $k \rightarrow \infty$ of both sides gives the
  result.
\end{proof}

We can then integrate the estimate of the last lemma to get an
analogous result for $\TM$.

\begin{lemma}\label{lem:6}
  Let $Y \subseteq M$ be measurable.  Then the function $\Vol(Y,
  \sdot) : \Mf \rightarrow \R$ mapping $\tilde{g} \mapsto \Vol(M,
  \tilde{g})$ is Lipschitz continuous with respect to $\Theta_Y$.  In
  particular, if $g_0, g_1 \in \Mf$, then
  \begin{equation*}
    \abs{\Vol(M, g_1) - \Vol(M, g_0)} \leq \frac{\sqrt{n}}{2}
    \Theta_Y(g_0, g_1) \leq \frac{\sqrt{n}}{2}
    \Theta_M(g_0, g_1).
  \end{equation*}
\end{lemma}
\begin{proof}
  By \cite[Lemma
  4.10]{clarke:_compl_of_manif_of_rieman_metric}, we have
  that
  \begin{equation*}
    \abs{\sqrt{\det G_1(x)} - \sqrt{\det G_0(x)}} \leq
    \frac{\sqrt{n}}{2} \theta^g_x(g_0(x), g_1(x)).
  \end{equation*}
  Using this, we can estimate
  \begin{equation*}
    \begin{aligned}
      \abs{\Vol(M, g_1) - \Vol(M, g_0)} &=
      \abs{\integral{Y}{}{}{d\mu_{g_1}} -
        \integral{Y}{}{}{d\mu_{g_0}}} \\
      &= \abs{\integral{Y}{}{ \left(
            \sqrt{\det G_1(x)} - \sqrt{\det G_0(x)}
          \right)}{d\mu_g}} \\
      &\leq \integral{Y}{}{\abs{\sqrt{\det G_1(x)} - \sqrt{\det
            G_0(x)}}}{d \mu_g} \\
      &\leq \frac{\sqrt{n}}{2} \integral{Y}{}{\theta^g_x(g_0(x),
        g_1(x))}{d\mu_g} \\
      &= \frac{\sqrt{n}}{2} \Theta_Y(g_0, g_1) \leq \frac{\sqrt{n}}{2}
      \Theta_M(g_0, g_1).
    \end{aligned}
  \end{equation*}
\end{proof}

\subsection{The completion of $(\M, \TM)$}\label{sec:completion-m-tm}

The above lemma suggests a strong parallel with the behavior of
$d$---again, compare Lemma \ref{lem:3}.  In fact, the next two
theorems will give us even stronger parallels, as we will see that the
completion of $(\M, \TM)$ can be identified with a quotient space of
$\Mf$.  We begin with a proof that to each $\TM$-Cauchy sequence, we
can associate a limit semimetric in $\Mm$.

Before we prove these theorems, let us first remark that for a
sequence $\{g_k\} \subset \Mm$, we define $\tgx$-convergence in
measure to $g_0 \in \Mm$ analogously to how we defined it for
$\absgx{\, \cdot \,}$.  Again, if the measure is not explicitly
mentioned, then $\mu_g$ is implied.

\begin{theorem}\label{thm:4}
  Let $\{g_k\} \subset \Mf$ be a $\Theta_M$-Cauchy sequence.  Then
  there exists an element $[g_0] \in \Mmhat$ such that $g_k
  \xrightarrow{\Theta_M} [g_0]$.  (In particular, $\TM(g_k, [g_0])$ is
  well-defined and finite for each $k \in \N$.)  Furthermore, if $g_0
  \in [g_0]$ is any representative, then we have that $\{g_k\}$
  $\theta^g_x$-converges to $g_0$ in measure, and $X_{g_0} =
  D_{\{g_k\}}$ (cf.~Definition \ref{dfn:7}) up to a nullset.

  Finally, if $\{\tilde{g}_k\} \subset \Mf$ is any $\TM$-Cauchy
  sequence that $\tgx$-converges to $\tilde{g}_0 \in \Mm$ in measure,
  then $\tilde{g}_k \xrightarrow{\Theta_M} \tilde{g}_0$.
\end{theorem}
\begin{proof}
  We begin with the first statement.  Since $\{g_k\}$ is a
  $\TM$-Cauchy sequence, it suffices to prove convergence for a
  subsequence.  By passing to a subsequence, we may assume that
  \begin{equation}\label{eq:18}
    \sum_{k = 1}^\infty \TM(g_k, g_{k+1}) < \infty.
  \end{equation}

  Now, for all $\delta > 0$ and $k, l \in \N$, let $E^{k,l}_\delta :=
  \{ x \in M \mid \theta^g_x(g_k, g_l) \geq \delta \}$.  For all
  $\epsilon > 0$, we can find $k_0 \in \N$ such that if $k, l \geq
  k_0$, then $\Vol(E^{k,l}_\delta, g) < \epsilon$.  (In other words,
  $\{g_k\}$ is $\theta^g_x$-Cauchy in measure.)  A straightforward
  argument shows that by again passing to a subsequence, we can assume
  that $\{g_k\}$ is $\theta^g_x$-Cauchy a.e.  Therefore, by Theorem
  \ref{thm:10}, for almost every $x \in M$, $\{g_k(x)\}$
  $\tgx$-converges to some element $[a_x] \in \overline{\Matx}$.  For
  each $x \in M$, let $g_0(x) := a_x$, where $a_x \in [a_x]$ is any
  representative.  (Note that $g_0$ is well-defined up to a set of
  measure zero, where we may without consequence set it equal to
  zero.)  We claim that $[g_0]$ is the desired limit element.
  
  Choose any representative $g_0 \in [g_0]$.  (Note that the choice of
  representative does not affect the quantity $\TM(g_k, g_0)$.)  Then
  for a.e.~$x \in M$ and $k \in \N$, $\tgx(g_k(x), g_0(x))$ is finite,
  positive, and independent of our choice of the representative $g_0$.
  Fix some $k \in \N$.  For each $l \in \N$, define functions by
  $f_{k,l}(x) := \tgx(g_k(x), g_l(x))$, and also define $f_k(x) :=
  \tgx(g_k(x), g_0(x))$.  Then by construction, we have that $\lim_{l
    \rightarrow \infty} f_{k,l}(x) = f_k(x)$ for a.e.~$x \in M$.
  Furthermore, if we define
  \begin{equation*}
    \alpha(x) := \sum_{m=1}^\infty \tgx(g_m, g_{m+1}),
  \end{equation*}
  then by the triangle inequality,
  \begin{equation*}
    \abs{f_{k,l}(x)} \leq \sum_{m=k}^{l-1} \tgx(g_m, g_{m+1}) \leq
    \alpha(x).
  \end{equation*}
  On the other hand, we claim that $\alpha \in L^1(M, g)$, since
  \begin{equation*}
    \integral{M}{}{\alpha}{d \mu_g} = \integral{M}{}{\sum_{m=1}^\infty
      \tgx(g_m, g_{m+1})}{d \mu_g} = \sum_{m=1}^\infty \TM(g_m,
    g_{m+1}) < \infty.
  \end{equation*}
  (Note here that we have used the assumption \eqref{eq:18}, and that
  we have implicitly exchanged an infinite sum and an integral in the
  second equality.  The latter is justified by an application of the
  Monotone Convergence Theorem of Lebesgue and Levi
  \cite[Thm.~2.8.2]{bogachev07:_measur_theor}---see the proof of Lemma
  4.17 in \cite{clarked):_compl_of_manif_of_rieman} for the full
  details of this argument.)
  
  Since for a.e.~$x \in M$, $\lim_{l \rightarrow \infty} f_{k,l}(x) =
  f_k(x)$, $\abs{f_{k,l}(x)} \leq \alpha(x)$, and $\alpha \in L^1(M,
  g)$, the Lebesgue Dominated Convergence Theorem applies to give that
  $f_k \in L^1(M, g)$ and
  \begin{equation*}
    \integral{M}{}{f_k}{d \mu_g} = \lim_{l \rightarrow \infty} \integral{M}{}{f_{k,l}}{d \mu_g}.
  \end{equation*}
  In other words, $\lim_{l \rightarrow \infty} \TM(g_k, g_l) =
  \TM(g_k, g_0) < \infty$.  In particular, $\TM(g_k, g_0)$ is
  well-defined and finite, as claimed in the theorem.

  To see that $\lim_{k \rightarrow \infty} \TM(g_k, g_0) = 0$, one
  must apply essentially the same argument to the sequence of
  functions $f_k(x) = \tgx(g_k(x), g_0(x))$ from above with the limit
  function $0$.  The function that bounds each $f_k(x)$ is
  \begin{equation*}
    \tgx(g_0(x), g_1(x)) + \sum_{m = 1}^\infty \tgx(g_m(x), g_{m+1}(x)),
  \end{equation*}
  which we claim is in $L^1(M, g)$.  For, by a special case of the
  preceding argument, $\TM(g_1, g_0) < \infty$, which is equivalent to
  saying that the first term is $L^1$.  But this already implies the
  claim, since the sum is just $\alpha(x)$, which we saw was $L^1$ in
  the preceding argument.

  Thus, we now have that $g_k \xrightarrow{\TM} [g_0]$, or in other
  words, $\tgx(g_k(x), g_0(x)) \xrightarrow{L^1(M, g)} 0$, implying
  that $\tgx(g_k(x), g_0(x)) \rightarrow 0$ in measure by Theorem
  \ref{thm:11}.  But this is exactly the assertion that $\{g_k\}$
  $\tgx$-converges to $g_0$ in measure.  From this, and \cite[Lemma
  4.10]{clarke:_compl_of_manif_of_rieman_metric}, one can
  also deduce that $D_{\{g_k\}} = X_{g_0}$ up to a nullset.

  Finally, let $\{\tilde{g}_k\} \subset \Mf$ be any $\TM$-Cauchy
  sequence that $\tgx$-converges to $\tilde{g}_0 \in \Mm$ in measure.
  We know that $\{\tilde{g}_k\}$ $\TM$-converges to some limit
  $\bar{g}_0 \in \Mm$, and that $\{\tilde{g}_k\}$ $\tgx$-converges to
  $\bar{g}_0$ in measure.  So it is not hard to see that off of the
  set where both $\tilde{g}_0(x)$ and $\bar{g}_0(x)$ are degenerate
  (that is, off of the set where it is possible that $\tilde{g}_0(x)
  \neq \bar{g}_0(x)$ but $\theta^g_x(\tilde{g}_0(x), \bar{g}_0(x)) =
  0$), $\tilde{g}_0$ and $\bar{g}_0$ must coincide a.e.  In other
  words, $[\tilde{g}_0] = [\bar{g}_0]$.  This implies that
  $\TM(\tilde{g}_k, \tilde{g}_0) = \TM(\tilde{g}_k, \bar{g}_0)$, and
  so $\tilde{g}_k \rightarrow \tilde{g}_0$, as was to be shown.
\end{proof}

Knowing now that the $\TM$-limit of a Cauchy sequence can be
identified with an element of $\Mm$, we demonstrate that this limit
must have finite volume, and thus actually lies in $\Mf$.

\begin{theorem}\label{thm:5}
  Let $\{g_k\} \subset \Mf$ $\Theta_M$-converge to $g_0 \in \Mm$.
  Then in fact $g_0 \in \Mf$ and the following hold:
  \begin{enumerate}
  \item\label{item:5} We have
    \begin{equation*}
      \left(
        \frac{\mu_{g_k}}{\mu_g}
      \right) \xrightarrow{L^1(M, g)}
      \left(
        \frac{\mu_{g_0}}{\mu_g}
      \right).
    \end{equation*}
    In particular, $\{ ( \mu_{g_\alpha} / \mu_g ) \mid \alpha = 0, 1,
    2, \dots \}$ is uniformly absolutely continuous with respect to
    $\mu_g$.
  \item\label{item:6} $\mu_{g_k}$ converges uniformly to $\mu_{g_0}$.
  \end{enumerate}
\end{theorem}
\begin{proof}
  We first prove statement \eqref{item:5}.  By Lemma \ref{lem:16}, we
  have
  \begin{equation*}
    \begin{aligned}
      \integral{M}{}{\abs{\left( \frac{\mu_{g_0}}{\mu_g} \right) -
          \left( \frac{\mu_{g_k}}{\mu_g} \right)}}{d \mu_g} &\leq
      \frac{\sqrt{n}}{2} \integral{M}{}{\tgx(g_k(x), g_0(x))}{d \mu_g}
      = \TM(g_k, g_0).
    \end{aligned}
  \end{equation*}
  The statement follows from this immediately.  Note that this also
  implies that we haveo $(\mu_{g_0} / \mu_g) \in L^1(M, g)$, or in
  other words, $g_0$ has finite total volume.  The uniform absolute
  continuity of $\{ ( \mu_{g_\alpha} / \mu_g ) \mid \alpha = 0, 1, 2,
  \dots \}$ is given by Theorem \ref{thm:11}.

  Statement \eqref{item:6} then follows from statement \eqref{item:5}
  and Lemma \ref{lem:18}.
\end{proof}

The above result implies, as noted, that $\overline{(\M, \TM)}$ is a
quotient space of some subspace of $\Mf$.  In fact, it is the same
space as $\overline{(\M, d)}$:

\begin{theorem}\label{thm:12}
  The completion of $\M$ with respect to $\TM$ can be naturally
  identified with $\Mfhat$.  This map is an isometry if we define
  $\TM$ on $\Mfhat$ by
  \begin{equation*}
    \TM([g_0], [g_1]) = \integral{M}{}{\tgx(g_0(x), g_1(x))}{d \mu_g}
  \end{equation*}
  for any $g_0, g_1 \in \Mf$.
\end{theorem}
\begin{proof}
  By Theorems \ref{thm:4} and \ref{thm:5}, any $\TM$-Cauchy sequence
  in $\M$ $\TM$-converges to an element of $\Mf$.  Furthermore, if
  $g_0 \in \Mf$ is any element, then there exists a $d$-Cauchy
  sequence that $d$-converges to $g_0$ (cf.~Theorem \ref{thm:6}).  By
  the estimate of Proposition \ref{prop:6}, one can see that this
  sequence is also $\TM$-Cauchy.  Thus, any element of $\Mf$ arises as
  the $\TM$-limit of some sequence in $\M$.  This proves that
  $\overline{(\M, \TM)}$ is a quotient space of $\Mf$ given by
  identifying all elements with distance zero from one another.
  Furthermore, Proposition \ref{prop:6} implies the formula
  \begin{equation}\label{eq:34}
    \TM(g_0, g_1) = \integral{M}{}{\tgx(g_0(x), g_1(x))}{d \mu_g}
  \end{equation}
  for any $g_0, g_1 \in \Mf$.

  But from (\ref{eq:34}) and the fact that $\tgx$ is a metric on
  $\Matx$ and zero on $\partial \Matx$, it is easy to see that if
  $g_0, g_1 \in \Mf$, then $g_0 \sim g_1$ if and only if $\TM(g_0,
  g_1) = 0$.  Thus the desired quotient space of $\Mf$ is exactly
  $\Mfhat = \Mf / {\sim}$.
\end{proof}

We now have a good understanding of the metric $\TM$, but before we
leave this section, we will prove a result that will come in useful
later.  It is based on the following pointwise result, which was
already known.

\begin{proposition}[{\cite[Prop.~4.13]{clarke:_compl_of_manif_of_rieman_metric}}]\label{prop:7}
  Let $a_0, a_1 \in \M_x$.  Then there exists a constant $C'(n)$,
  depending only on $n$, such that
  \begin{equation*}
    \theta^g_x(a_0, a_1) \leq C'(n) \left( \sqrt{\det A_0} +
      \sqrt{\det A_1} \right).
  \end{equation*}
\end{proposition}

By integrating this inequality, we get an estimate for $\TM$ that is
of exactly the same form as Proposition \ref{prop:1}.

\begin{proposition}\label{prop:4}
  Let $g_0, g_1 \in \Mf$, and let $A := \carr(g_1 - g_0)$.  Then there
  exists a constant $C'(n)$, depending only on $n = \dim M$, such that
  \begin{equation}\label{eq:3}
    \TM(g_0, g_1) \leq C'(n)
    \left(
      \Vol(A, g_0) + \Vol(A, g_1)
    \right).
  \end{equation}
\end{proposition}
\begin{proof}
  We claim that the inequality of Proposition \ref{prop:7} still holds
  if $a_0, a_1 \in \cl(\Matx)$.  For if both lie in $\partial \Matx$,
  then by Theorem \ref{thm:10} we have that $\tgx(a_0, a_1) = 0$, so
  the statement is vacuous.  If one (say $a_0$) is in $\partial
  \Matx$, then we choose a sequence $a_0^k \xrightarrow{\tgx} a_0$,
  where all $a_0^k \in \Matx$.  By Lemma \ref{lem:16}, $\sqrt{\det
    A_0^k} \rightarrow \sqrt{\det A_0} = 0$.  On the other hand, for
  each $k \in \N$ we have
  \begin{equation*}
    \tgx(a_0^k, a_1) \leq C'(n)
    \left(
      \sqrt{\det A_0^k} + \sqrt{\det A_1}
    \right),
  \end{equation*}
  so taking the limit of the above inequality proves the claim.

  Finally, we note that since $\sqrt{\det G_0} = (\mu_{g_0} / \mu_g)$
  and $\sqrt{\det G_1} = (\mu_{g_1} / \mu_g)$, the inequality
  \eqref{eq:3} follows immediately from integrating the pointwise
  estimate of Proposition \ref{prop:7} after substituting $a_0 :=
  g_0(x)$ and $a_1 := g_1(x)$.
\end{proof}

We now have all of the background results that we need and are ready
to move into the main body of the paper.

\section{Convergence results}\label{sec:convergence-mf}

In this section, we will begin by proving the equivalence of the
topologies of $d$ and $\TM$ on a given quasi-amenable subset.  We will
then use this to extend the equivalence to all of $\Mf$.  Using that,
we will finally relatively quickly arrive at a homeomorphism between
$\overline{(\M, d)}$ and $\overline{(\M, \TM)}$.

\subsection{The topology on quasi-amenable
  subsets}\label{sec:topol-quasi-amen}

We begin this subsection with a straightforward extension of Theorem
\ref{thm:7} to degenerate metrics.

\begin{lemma}\label{lem:4}
  Let a quasi-amenable subset $\U \subset \M$ and $\epsilon > 0$ be
  given.  Then there exists $\delta > 0$ such that if $g_0, g_1 \in
  \U^0$ (the $L^2$-closure of $\U$) with $\norm{g_1 - g_0}_g <
  \delta$, we have $d(g_0, g_1) < \epsilon$.
\end{lemma}
\begin{proof}
  Let $\delta > 0$ be the number, guaranteed by Theorem \ref{thm:7},
  for which the following holds: For all $\tilde{g}_0, \tilde{g}_1 \in
  \U$ with $\norm{\tilde{g}_1 - \tilde{g}_0}_g < 2 \delta$, we have
  $d(\tilde{g}_0, \tilde{g}_1) < \epsilon$.

  We claim that this is the desired number $\delta$.  To see this, let
  $g_0, g_1 \in \U^0$ with $\norm{g_1 - g_0}_g < \delta$, and choose
  sequences $\{g_0^k\}$ and $\{g_1^k\}$ in $\U$ that both $L^2$- and
  $d$-converge to $g_0$ and $g_1$, respectively.  (The existence of
  such sequences is assured by Proposition \ref{prop:5}.)  Then by
  definition,
  \begin{equation}\label{eq:12}
    d(g_0, g_1) = \lim_{k \rightarrow \infty} d(g_0^k, g_1^k).
  \end{equation}
  On the other hand, since $g_i^k \xrightarrow{L^2} g_i$ for $i = 0,
  1$, we have
  \begin{equation*}
    \lim_{k \rightarrow \infty} \norm{g_1^k - g_0^k}_g = \norm{g_1 - g_0}_g.
  \end{equation*}
  Since $\norm{g_1 - g_0}_g < \delta$, this implies that $\norm{g_1^k
    - g_0^k}_g < 2 \delta$ for $k$ large enough, and so by the
  assumption on $\delta$, $d(g_0^k, g_1^k) < \epsilon$ for $k$ large.
  But then \eqref{eq:12} implies that $d(g_0, g_1) < \epsilon$, as was
  to be proved.
\end{proof}

Next, we need a lemma that allows us to compare open balls in the
metric $\tgx$ to open balls in the norm $\absgx{\, \cdot \,}$.  It is
possible to do this uniformly if we restrict to compact subsets in
$\Matx$; in particular, we will need those subsets given in the next
lemma.

\begin{lemma}\label{lem:7}
  Let any numbers $\zeta, \tau > 0$ be given.  For each $x \in M$, we
  define
  \begin{equation*}
    \Matx^{\zeta,\tau} := \{ a \in \Matx \mid \sqrt{\det A}
    \geq \zeta,\ \abs{a_{ij}} \leq \tau\ \textnormal{for all}\ 1 \leq
    i, j \leq n \} \subset \Matx.
  \end{equation*}
  Furthermore, for any $\lambda \geq 0$ and $a \in \Matx$, denote by
  $B^{\theta^g_x}_{a}(\lambda)$ and $B^{\absgx{\, \cdot
      \,}}_{a}(\lambda)$ the open balls of radius $\lambda$ around $a$
  with respect to $\tgx$ and $\absgx{\, \cdot \,}$, respectively.

  For each $x \in M$, $a \in \Matx$, and $\kappa > 0$, we also define
  a function
  \begin{equation*}
    \eta_{x,a}(\kappa) := \sup \left\{ \lambda \in \R \midmid
      B^{\theta^g_x}_{a}(\lambda) \subset B^{\absgx{\, \cdot \,}}_{a}(\kappa) \right\}.
  \end{equation*}

  Then $\eta_{x,a}$ takes values in $(0, \infty)$, as does the
  function
  \begin{equation}\label{eq:27}
    \eta(\kappa) := \inf_{x \in M,\ a \in \Matx^{\zeta,\tau}}
    \eta_{x,a}(\kappa).
  \end{equation}
\end{lemma}
\begin{proof}
  Since $\Matx$ is a finite-dimensional manifold, the topologies
  induced by $\absgx{\, \cdot \,}$ and $\theta^g_x$ coincide.  This
  implies, in particular, that for all $\kappa > 0$ and $a \in \Matx$,
  we can find $\lambda > 0$ such that
  \begin{equation*}
    B^{\theta^g_x}_{a}(\lambda) \subset B^{\absgx{\, \cdot \,}}_{a}(\kappa).
  \end{equation*}
  This also implies that the supremum of such $\lambda$ must be
  finite, proving that $\eta_{x,a}$ takes values in $(0, \infty)$.
  
  To see that $\eta$ is also a positive, finite function, we note that
  $\langle \cdot, \cdot \rangle^0_a$ and $\langle \cdot , \cdot
  \rangle_{g(x)}$ depend smoothly on $x$ and $a$.  Thus, $\eta_{x,a}$
  is continuous separately in $x$ and $a$.  Therefore the result
  follows from the compactness of $\Matx^{\zeta, \tau}$ and $M$.
\end{proof}

Using this lemma, we can get a bound on the distance between two
elements of a quasi-amenable subset if we have a uniform, pointwise
bound on their distance in $\tgx$.  This, of course, will not help us
much when trying to prove $d$-convergence of a sequence $\{g_k\}$ that
$\TM$-converges to $g_0$, as such a sequence only has $\tgx(g_k(x),
g_0(x))$ converging to zero in $L^1(M, g)$.  However, it will be
sufficiently strong to facilitate a cut-off argument.  These arguments
will be a recurring theme in the remaining proofs.

\begin{lemma}\label{lem:5}
  Let $\U \subset \M$ be any quasi-amenable subset, and let $\epsilon
  > 0$ be given.  Then there exists $\delta > 0$ such that if $g_0 \in
  \U^0$, $g_1 \in \Mf$, and $\theta^g_x(g_0(x), g_1(x)) < \delta$ for
  all $x \in M$, we have $d(g_0, g_1) < \epsilon$.
\end{lemma}
\begin{proof}
  Let $\epsilon > 0$ be given.  The idea is to cut off $g_0$ and $g_1$
  by setting them equal to zero on a small subset such that the
  cut-off semimetrics belong to a common quasi-amenable subset.  The
  distance from the cut-off semimetrics to the original semimetrics
  will then be small by Proposition \ref{prop:1}.  The distance in the
  $L^2$ norm between the two cut-off semimetrics can be estimated
  using Lemma \ref{lem:7}, which gives us a bound on the $d$-distance
  by Lemma \ref{lem:4}.
  
  To fill in the details of this, we first define a positive constant
  $\zeta$ by
  \begin{equation}\label{eq:13}
    \zeta := \frac{\epsilon^2}{32 C(n)^2 \Vol(M, g)}
  \end{equation}
  where $C(n)$ is the constant from Proposition \ref{prop:1}.

  Let $\tau$ be the number such that the quasi-amenable subset $\U$ is
  given by
  \begin{equation*}
    \{ g \in \M \mid \abs{g_{ij}(x)} \leq \tau\ \textnormal{for all}\ 1 \leq
    i, j \leq n,\ x \in M \}.
  \end{equation*}
  As in the last lemma, for each $x \in M$, we then consider the set
  \begin{equation*}
    \Matx^{\zeta, \tau} := \{ a \in \Matx \mid \sqrt{\det A}
    \geq \zeta,\ \abs{a_{ij}} \leq \tau\ \textnormal{for all}\ 1 \leq
    i, j \leq n \} \subset \Matx,
  \end{equation*}
  We also define the positive function $\eta$ as in the last lemma.

  Next, we denote by $\tilde{\U}$ the ``double'' of $\U$, i.e.,
  \begin{equation*}
    \tilde{\U} := \{ g \in \M \mid \abs{g_{ij}(x)} \leq 2\tau\ \textnormal{for all}\ 1 \leq
    i, j \leq n,\ x \in M \}.
  \end{equation*}
  Then it is clear that there exists $\alpha > 0$ such that if
  $\tilde{g}_0 \in \U^0$, $\tilde{g}_1 \in \Mf$, and
  \begin{equation*}
    \absgx{\tilde{g}_1(x) - \tilde{g}_0(x)} < \alpha \quad \textnormal{for all}\ x \in
    M,
  \end{equation*}
  then $\tilde{g}_1 \in \tilde{\U}^0$.
  
  Let $\kappa$ now be the constant from Lemma \ref{lem:4} such that
  $\tilde{g}_0, \tilde{g}_1 \in \tilde{\U}^0$ and $\norm{\tilde{g}_1 -
    \tilde{g}_0}_g < \kappa$ implies $d(\tilde{g}_0, \tilde{g}_1) <
  \epsilon/2$.  Then we claim that
  \begin{equation}\label{eq:11}
    \delta = \min \left\{ \eta(\alpha), \eta \left( \frac{\kappa}{\sqrt{\Vol(M, g)}}
      \right),  \frac{\epsilon^2}{16 \sqrt{n} C(n)^2 \Vol(M, g)}
    \right\}
  \end{equation}
  is the desired constant.

  To see this, let $g_0, g_1 \in \U^0$ with $\theta^g_x(g_0(x),
  g_1(x)) < \delta$ be given.  Furthermore, let
  \begin{equation*}
    E^\zeta :=
    \left\{
      x \in M \mid g_0(x) \in \Matx^{\zeta,\tau}
    \right\}.
  \end{equation*}
  We then define
  \begin{equation*}
    g_0^\zeta := \chi(E^\zeta) g_0
    \quad \textnormal{and} \quad
    g_1^\zeta := \chi(E^\zeta) g_1.
  \end{equation*}

  Since $(\mu_{g_0} / \mu_g) < \zeta$ on $M \setminus E^\zeta$, we
  have that
  \begin{equation*}
    \Vol(M \setminus E^\zeta, g_0) = \integral{E^\zeta}{}{
      \left(
        \frac{\mu_{g_0}}{\mu_g}
      \right)}{d \mu_g} < \zeta \integral{E^\zeta}{}{}{d \mu_g} \leq
    \zeta \cdot \Vol(M, g).
  \end{equation*}

  Additionally, since $\theta^g_x(g_0(x), g_1(x)) < \delta$ for all $x
  \in M$, $\Theta_M(g_0, g_1) < \delta \cdot \Vol(M, g)$.  Therefore,
  by Lemma \ref{lem:6},
  \begin{equation*}
    \Vol(M \setminus E^\zeta, g_1) \leq \Vol(M \setminus E^\zeta, g_0) + \frac{\sqrt{n}}{2}
    \Theta_M(g_0, g_1) < \zeta \cdot \Vol(M, g) + \frac{\sqrt{n}}{2} \delta \cdot \Vol(M, g).
  \end{equation*}
  
  By Proposition \ref{prop:1}, and using that $g_0$ and $g_0^\zeta$
  differ only on $M \setminus E^\zeta$, where $g_0^\zeta = 0$, we thus
  see that
  \begin{equation*}
    d(g_0, g_0^\zeta) \leq C(n) \sqrt{\Vol(M \setminus
      E^\zeta, g_0)} < C(n) \sqrt{\zeta \cdot \Vol(M, g)}.
  \end{equation*}
  Similarly,
  \begin{equation*}
    d(g_1, g_1^\zeta) \leq C(n) \sqrt{\Vol(M \setminus E^\zeta, g_1)} < C(n) \sqrt{\zeta \cdot
      \Vol(M, g) + \frac{\sqrt{n}}{2} \delta \cdot \Vol(M, g)}.
  \end{equation*}
  Inserting \eqref{eq:13} and \eqref{eq:11} (for $\delta$, use the
  last of the values in the minimum) into the above estimate and
  simplifying gives
  \begin{equation}\label{eq:14}
    d(g_0, g_0^\zeta) + d(g_1, g_1^\zeta) < \epsilon/2.
  \end{equation}

  We now wish to estimate $d(g_0^\zeta, g_1^\zeta)$, so that we can
  use \eqref{eq:14} and the triangle inequality to estimate $d(g_0,
  g_1)$.  To do so, note that $g_0^\zeta$ and $g_1^\zeta$ differ only
  on $E^\zeta$, where $g_0(x) \in \Matx^{\zeta,\tau}$.  But we also
  know that
  \begin{equation*}
    \theta^g_x(g_0^\zeta(x), g_1^\zeta(x)) < \delta \leq \eta(\alpha),
  \end{equation*}
  meaning
  \begin{equation*}
    \absgx{g_1^\zeta(x) - g_0^\zeta(x)} < \alpha.
  \end{equation*}
  By our choice of $\alpha$, this implies that $g_1^\zeta \in
  \tilde{\U}^0$.

  Finally, we recall that
  \begin{equation*}
    \theta^g_x(g_0^\zeta(x), g_1^\zeta(x)) < \delta \leq \eta \left(
      \frac{\kappa}{\sqrt{\Vol(M, g)}} \right)
  \end{equation*}
  for all $x \in M$.  By the definition of $\eta$, this immediately
  implies that
  \begin{equation*}
    \absgx{g_1^\zeta(x) - g_0^\zeta(x)} < \frac{\kappa}{\sqrt{\Vol(M,g)}}.
  \end{equation*}
  Therefore,
  \begin{equation*}
    \norm{g_1^\zeta - g_0^\zeta}_g  =
    \left(
      \integral{M}{}{\absgx{g_1^\zeta(x) - g_0^\zeta(x)}^2}{d \mu_g}
    \right)^{1/2} < \left(
      \frac{\kappa^2}{\Vol(M, g)} \integral{M}{}{}{d \mu_g}
    \right)^{1/2} = \kappa.
  \end{equation*}
  But by our choice of $\kappa$, and since as we noted, $g_0^\zeta,
  g_1^\zeta \in \tilde{\U}^0$, the above inequality implies that
  $d(g_0^\zeta, g_1^\zeta) < \epsilon / 2$.  This, combined with
  \eqref{eq:14} and the triangle inequality, gives the desired result.
\end{proof}

With this long estimate out of the way, we can show the equivalence of
the topologies of $d$ and $\TM$ on a quasi-amenable subset.  In fact,
the following proposition is even more general, as it only requires
that the limit semimetric be bounded---i.e., lie in some
quasi-amenable subset.  The sequence converging to this limit can have
elements anywhere in $\Mf$.

\begin{proposition}\label{prop:3}
  Let $g_0 \in \Mf$ be a bounded semimetric, and let $\{g_k\} \subset
  \Mf$ be any sequence.  Then $g_k \xrightarrow{d} g_0$ if and only if
  $g_k \xrightarrow{\Theta_M} g_0$.
\end{proposition}
\begin{proof}
  That $g_k \xrightarrow{d} g_0$ implies $g_k \xrightarrow{\Theta_M}
  g_0$ is clear from
  \cite[Prop.~4.25]{clarke:_compl_of_manif_of_rieman_metric},
  so we turn to the converse statement.

  Let $\epsilon > 0$ be given, and let $\delta > 0$ be the number
  guaranteed by Lemma \ref{lem:5}.  We then define
  \begin{equation*}
    E^\delta_k := \left\{
      x \midmid \theta^g_x(g_k(x), g_0(x)) \geq \delta
    \right\}
  \end{equation*}
  and
  \begin{equation*}
    g^\delta_k := \chi(E^\delta_k) g_0 + \chi(M \setminus E^\delta_k) g_k.
  \end{equation*}

  The first thing we see is that $g^\delta_k$ differs from $g_0$ only
  on $M \setminus E^\delta_k$, and that for $x \in M \setminus
  E^\delta_k$, we have $\theta^g_x(g^\delta_k(x), g_0(x)) < \delta$.
  Therefore, by Lemma \ref{lem:5} and the choice of $\delta$, we
  immediately get
  \begin{equation}\label{eq:10}
    d(g^\delta_k, g_0) < \epsilon.
  \end{equation}
  (Note that the boundedness of $g_0$ implies that there exists a
  quasi-amenable subset $\U$ with $g_0 \in \U^0$, so that Lemma
  \ref{lem:5} indeed applies.)
  
  We now claim that
  \begin{equation}\label{eq:8}
    \begin{aligned}
      d(g_k, g_k^\delta) &\leq C(n) \left( \sqrt{\Vol(E^\delta_k,
          g_k)} + \sqrt{\Vol(E_k^\delta, g_k^\delta)}
      \right) \\
      &= C(n) \left( \sqrt{\Vol(E^\delta_k, g_k)} +
        \sqrt{\Vol(E_k^\delta, g_0)} \right).
    \end{aligned}
  \end{equation}
  The first line follows from Proposition \ref{prop:1}.  The second
  line follows because $g_k^\delta$ coincides with $g_0$ on
  $E_k^\delta$.
  
  Now, because of Lemma \ref{lem:6}, we can choose $k$ large enough
  that $\Vol(E^\delta_k, g_k) \leq \Vol(E^\delta_k, g_0) +
  \epsilon^2$.  Furthermore, by Theorem \ref{thm:11}, $g_k
  \xrightarrow{\Theta_M} g_0$ implies that the function
  $\theta^g_x(g_k(x), g_0(x))$ on $M$ converges to zero in
  $\mu_g$-measure.  By Lemma \ref{lem:14}, the convergence also holds
  in $\mu_{g_0}$-measure, so we can choose $k$ large enough that
  $\Vol(E^\delta_k, g_0) < \epsilon^2$.  This together with
  (\ref{eq:8}) implies that for $k$ large enough,
  \begin{equation}\label{eq:9}
    d(g_k, g_k^\delta) \leq C(n) \cdot (\sqrt{2} + 1) \epsilon.
  \end{equation}
  
  Thus, \eqref{eq:10}, \eqref{eq:9}, and the triangle inequality imply
  that for $k$ large enough,
  \begin{equation*}
    d(g_k, g_0) < C(n) \cdot (\sqrt{2} + 2) \epsilon.
  \end{equation*}
  Since $\epsilon$ was arbitrary, the fact that $d(g_k, g_0)
  \rightarrow 0$ follows.
\end{proof}

\subsection{The topology on $\Mf$ and the completions of
  $\M$}\label{sec:topol-mf-compl}

The next goal is to use the results of the last subsection to extend
Proposition \ref{prop:3} to allow the limit semimetric to be any
element of $\Mf$.  To do so, we will use a strategy similar to that of
the proof of
\cite[Thm.~5.14]{clarke:_compl_of_manif_of_rieman_metric}.
The basic idea is to reduce the case of an unbounded limit semimetric
to that of a bounded one by multiplying the sequence in question and
its limit with appropriate positive functions that tame the unbounded
parts of the limit semimetric.  If we do this carefully, then we can
use some arguments taking advantage of Proposition \ref{prop:8} to
control how far (with respect to $d$) these conformal changes move the
sequence and limit within $\Mf$.

To begin with, let's investigate how conformal changes affect the
$\TM$-distance between points.

\begin{lemma}\label{lem:8}
  Let $g_0, g_1 \in \Mf$, and let $\rho$ be a measurable, positive
  function with $\rho(x) \leq 1$ for all $x \in M$.  Then
  $\Theta_M(\rho g_0, \rho g_1) \leq \Theta_M(g_0, g_1)$.
\end{lemma}
\begin{proof}
  The statement would follow immediately if we showed that
  $\theta^g_x(\rho g_0(x), \rho g_1(x)) \leq \theta^g_x(g_0(x),
  g_1(x))$ for all $x \in M$.  So fix an arbitrary $x \in M$ and
  $\epsilon > 0$, and let $a_t$, for $t \in [0, 1]$, be any path in
  $\Matx$ with endpoints $g_0(x)$ and $g_1(x)$ such that
  \begin{equation*}
    L^0(a_t) \leq \theta^g_x(g_0(x), g_1(x)) + \epsilon,
  \end{equation*}
  where $L^0(a_t)$ denotes the length of $a_t$ as measured by $\langle
  \cdot, \cdot \rangle^0$.

  Since $\rho(x) a_t$ is a path from $\rho(x) g_0(x)$ to $\rho(x)
  g_1(x)$, we have
  \begin{equation*}
    \begin{aligned}
      \theta^g_x (\rho(x) g_0(x), \rho(x) g_1(x)) &\leq L^0(\rho(x)
      a_t) = \integral{0}{1}{\sqrt{\langle \rho(x) a'_t, \rho(x) a'_t
          \rangle^0_{\rho(x) a_t}}}{dt} \\
      &= \integral{0}{1}{\sqrt{\tr_{\rho(x) a_t}((\rho(x) a'_t)^2)
          \det(\rho(x) A_t)}}{dt}\\
      &= \integral{0}{1}{\rho(x)^{n/2}
        \sqrt{\tr_{a_t}((a'_t)^2) \det(A_t)}}{dt} \\
      &\leq \integral{0}{1}{\sqrt{\langle a'_t, a'_t
          \rangle^0_{a_t}}}{dt} = L^0(a_t) \leq \theta^g_x(g_0(x),
      g_1(x)) + \epsilon.
    \end{aligned}
  \end{equation*}
  Since $x$ and $\epsilon$ were arbitrary, we have the desired result.
\end{proof}

A simple consequence is that a $\TM$-convergent sequence that is
multiplied with such a function $\rho$ is still $\TM$-convergent.

Next, we need to (algebraically) extend the exponential mapping as
given by the expression in Proposition \ref{prop:8} to a more general
class of functions and basepoints.  In the following, we will make the
maximal such definition that still guarantees that the image of this
extended ``exponential mapping'' has finite volume if the basepoint
does.

\begin{lemma}\label{lem:17}
  For each $\tilde{g} \in \Mf$, define
  \begin{equation*}
    \mathcal{F}_{\tilde{g}} :=
    \left\{
      \zeta \in L^2(M, \tilde{g}) \midmid \zeta(x) \geq - \frac{4}{n}\
      \textnormal{for all}\ x \in M
    \right\}.
  \end{equation*}

  Then there exists a map
  \begin{equation}\label{eq:4}
    \begin{aligned}
      \psi_{\tilde{g}} : \mathcal{F}_{\tilde{g}} &\rightarrow \Mf \\
      \zeta &\mapsto \left( 1 + \frac{n}{4} \zeta \right)^{4/n}
      \tilde{g}.
    \end{aligned}
  \end{equation}
  (Note that $\psi_{\tilde{g}}(\zeta)$ is formally the same expression
  as $\exp_{\tilde{g}}(\zeta \tilde{g})$, if $\tilde{g} \in \M$ and
  $\zeta \in C^\infty(M)$ with $\zeta > - \frac{4}{n}$.)
\end{lemma}
\begin{proof}
  We have to prove that if $\zeta \in \mathcal{F}_{\tilde{g}}$, then
  $\psi_{\tilde{g}}(\zeta) \in \Mf$.  From (\ref{eq:4}) and the
  definition of $\mathcal{F}_{\tilde{g}}$, it is clear that
  $\psi_{\tilde{g}}(\zeta)$ is a semimetric, since the conformal
  factor in front of $\tilde{g}$ is nonnegative.

  To see that $\psi_{\tilde{g}}(\zeta)$ has finite volume, note that
  \begin{equation*}
    \mu_{\psi_{\tilde{g}}} = \left( 1 + \frac{n}{4} \zeta \right)^2
    \mu_{\tilde{g}} =
    \left(
      1 + \frac{n}{2} \zeta + \frac{n^2}{16}\zeta^2
    \right) \mu_{\tilde{g}}.
  \end{equation*}
  Thus, finite volume would follow if we could show that each summand
  in the parentheses on the right is in $L^1(M, \tilde{g})$.  But the
  constant function $1 \in L^2(M, \tilde{g})$ by finite volume of
  $\mu_{\tilde{g}}$.  Furthermore, $\zeta^2 \in L^1(M, \tilde{g})$
  since $\zeta \in L^2(M, \tilde{g})$.  Finally, as is well known (it
  is a simple consequence of Hölder's Inequality), finite volume of
  $(M, \tilde{g})$ implies that $L^2(M, \tilde{g}) \subset L^1(M,
  \tilde{g})$.  Therefore, $\zeta \in L^1(M, \tilde{g})$, completing
  what was to be shown.
\end{proof}

We will retain the notation $\mathcal{F}_{\tilde{g}}$ and
$\psi_{\tilde{g}}$ from the previous lemma for the remainder of the
section.

The following two lemmas will allow us to control the distance between
different conformal changes of a semimetric---one of the goals we
outlined at the beginning of this subsection.  We need to first
restrict to basepoint semimetrics that are bounded, and can then
extend this to the general case.

\begin{lemma}\label{lem:11}
  Let $\tilde{g} \in \Mf$ be a bounded, measurable semimetric, and let
  $\kappa, \lambda \in \mathcal{F}_{\tilde{g}}$.  Then
  \begin{equation*}
    d(\psi_{\tilde{g}}(\kappa), \psi_{\tilde{g}}(\lambda)) \leq \sqrt{n} \norm{\lambda - \kappa}_{\tilde{g}}.
  \end{equation*}
\end{lemma}
\begin{proof}
  By the proof of
  \cite[Thm.~5.14]{clarke:_compl_of_manif_of_rieman_metric},
  we can find sequences $\{\kappa_k\}$ and $\{\lambda_k\}$ of smooth
  functions with the following properties:
  \begin{itemize}
  \item $\{\kappa_k\}$ and $\{\lambda_k\}$ converge in $L^2(M,
    \tilde{g})$ to $\kappa$ and $\lambda$, respectively,
  \item $\kappa_k, \lambda_k > -\frac{4}{n}$ for all $k \in \N$, and
  \item we have
    \begin{equation*}
      \lim_{k \rightarrow \infty} d(\psi_{\tilde{g}}(\kappa_k), \psi_{\tilde{g}}(\kappa)) = 0 =
      \lim_{k \rightarrow \infty} d(\psi_{\tilde{g}}(\lambda_k), \psi_{\tilde{g}}(\lambda)).
    \end{equation*}
  \end{itemize}
  Furthermore, \cite[Lemma
  5.16]{clarke:_compl_of_manif_of_rieman_metric} implies
  \begin{equation*}
    d(\psi_{\tilde{g}}(\kappa_k), \psi_{\tilde{g}}(\lambda_k)) \leq \sqrt{n} \norm{\lambda_k -
      \kappa_k}_{\tilde{g}}.
  \end{equation*}
  Thus, the triangle inequality gives
  \begin{equation*}
    \begin{aligned}
      d(\psi_{\tilde{g}}(\kappa), \psi_{\tilde{g}}(\lambda)) &\leq
      \lim_{k \rightarrow \infty} [ d(\psi_{\tilde{g}}(\kappa),
      \psi_{\tilde{g}}(\kappa_k)) + d(\psi_{\tilde{g}}(\kappa_k),
      \psi_{\tilde{g}}(\lambda_k)) + d(\psi_{\tilde{g}}(\lambda_k), \psi_{\tilde{g}}(\lambda)) ] \\
      &\leq \lim_{k \rightarrow \infty} \sqrt{n} \norm{\lambda_k -
        \kappa_k}_{\tilde{g}} = \sqrt{n} \norm{\lambda -
        \kappa}_{\tilde{g}}.
    \end{aligned}
  \end{equation*}
\end{proof}

\begin{lemma}\label{lem:12}
  Let $\tilde{g} \in \Mf$ be any element, and let $\kappa, \lambda \in
  \mathcal{F}_{\tilde{g}}$.  Then
  \begin{equation*}
    d(\psi_{\tilde{g}}(\kappa), \psi_{\tilde{g}}(\lambda)) \leq \sqrt{n} \norm{\lambda - \kappa}_{\tilde{g}}.
  \end{equation*}
\end{lemma}
\begin{proof}
  Let $\xi$ be a positive, measurable function such that $\tilde{g}_0
  := \xi \tilde{g}$ is bounded.  For the remainder of the proof, we
  abbreviate $\psi := \psi_{\tilde{g}}$ and $\psi_0 :=
  \psi_{\tilde{g}_0}$.  Define
  \begin{equation*}
    \kappa_0 :=
    \xi^{-n/4} \left(
      \kappa + \frac{4}{n} - \xi^{n/4}
    \right) \qquad \textnormal{and} \qquad
    \lambda_0 := \xi^{-n/4} \left(
      \lambda + \frac{4}{n} - \xi^{n/4}
    \right).
  \end{equation*}
  Then a simple calculation shows that $\psi_0(\kappa_0) =
  \psi(\kappa)$ and $\psi(\lambda_0) = \psi(\lambda)$.
  
  But on the other hand, since $\tilde{g}_0$ is bounded, we can apply
  the previous lemma to get
  \begin{equation*}
    \begin{aligned}
      d(\psi_0(\kappa_0), \psi_0(\lambda_0)) &\leq \sqrt{n}
      \norm{\lambda_0 - \kappa_0}_{\tilde{g}_0} \\
      &= \sqrt{n} \left( \integral{M}{}{ \left( \xi^{-n/4} \left(
              \lambda + \frac{4}{n} - \xi^{n/4} \right) -
          \right. \right. \\
            &\qquad \qquad \left. \left. \xi^{-n/4} \left(
              \kappa + \frac{4}{n} - \xi^{n/4} \right) \right)^2}{d
          \mu_{\tilde{g}_0}} \right)^{1/2} \\
      &= \sqrt{n} \left( \integral{M}{}{ \left( \lambda - \kappa
          \right)^2 \xi^{-n/2}}{d \mu_{\tilde{g_0}}} \right)^{1/2} =
      \sqrt{n} \left( \integral{M}{}{ \left( \lambda - \kappa
          \right)^2}{d \mu_{\tilde{g}}}
      \right)^{1/2} \\
      &= \sqrt{n} \norm{\lambda - \kappa}_{\tilde{g}}.
    \end{aligned}
  \end{equation*}
  This gives the desired result.
\end{proof}

The last technical result that we will need for the moment concerns
the behavior of the norms of bounded functions with respect to a
$d$-convergent sequence of semimetrics.

\begin{lemma}\label{lem:9}
  Let $g_k, g_0 \in \Mf$ with $g_k \xrightarrow{\TM} g_0$, and let
  $\lambda$ be a bounded, measurable function on $M$.  Then $\lambda
  \in L^2(M, g_i)$ for all $i = 0, 1, 2, \dots$, and
  $\norm{\lambda}_{g_k} \rightarrow \norm{\lambda}_{g_0}$.
\end{lemma}
\begin{proof}
  We note that for $i = 0, 1, 2, \dots$,
  \begin{equation*}
    \norm{\lambda}_{g_i} = \norm{\lambda
      \sqrt{\left(
          \frac{\mu_{g_i}}{\mu_g}
        \right)}}_g = \left( \integral{M}{}{\lambda^2 \left(
        \frac{\mu_{g_i}}{\mu_g}
      \right)}{d \mu_g} \right)^{1/2}.
  \end{equation*}
  But by Theorem \ref{thm:5},
  \begin{equation*}
    \left(
      \frac{\mu_{g_k}}{\mu_g}
    \right) \xrightarrow{L^1(M, g)} \left(
      \frac{\mu_{g_0}}{\mu_g}
    \right).
  \end{equation*}
  The result then follows straightforwardly from this and from the
  boundedness of $\lambda$.
\end{proof}

With these results at hand, we can prove the equivalence of the
topologies of $d$ and $\TM$ on $\Mf$.

\begin{theorem}\label{thm:2}
  Let $g_k, g_0 \in \Mf$.  Then $g_k \xrightarrow{d} g_0$ if and only
  if $g_k \xrightarrow{\Theta_M} g_0$.
\end{theorem}
\begin{proof}
  As in Proposition \ref{prop:3}, the only statement that needs
  proving is that $g_k \xrightarrow{\Theta_M} g_0$ implies $g_k
  \xrightarrow{d} g_0$.

  So let $\epsilon > 0$ be given, and we wish to see that $d(g_k, g_0)
  < \epsilon$ for $k$ large enough.  The idea of the proof is to find
  a sequence $\{\sigma_l\}$ of measurable, positive functions with
  $\sigma_l \leq 1$ for all $l$, and with the property that $\sigma_l
  g_0$ is bounded.  Then Lemma \ref{lem:8} and Proposition
  \ref{prop:3} apply to give that $\sigma_l g_k \xrightarrow{d}
  \sigma_l g_0$, as $k \rightarrow \infty$, for each $l$.
  Furthermore, if we can arrange that $\sigma_l g_0 \xrightarrow{d}
  g_0$, and we can say that $d(\sigma_l g_k, g_k)$ is close to
  $d(\sigma_l g_0, g_0)$ for large $k$, then we have estimated each
  term on the right-hand side of this double application of the
  triangle inequality:
  \begin{equation*}
    d(g_k, g_0) \leq d(g_k, \sigma_l g_k) + d(\sigma_l g_k, \sigma_l
    g_0) + d(\sigma_l g_0, g_0).
  \end{equation*}
  So let us get down to the details of this argument.

  We first choose a measurable, positive function $\xi$ on $M$ such
  that $\xi \leq 1$ and $\xi g_0$ is bounded.  Set $\rho := \xi^{-1}$
  and $g_i^0 := \xi g_i$ for $i = 0, 1, 2, \dots$, so that $g_i = \rho
  g_i^0$.  A simple estimate using the finite volume of $g_0$ shows
  that $\rho \in L^{n/2}(M, g_0^0)$.

  For each $i = 0,1,2,\dots$, let's abbreviate $\psi_i :=
  \psi_{g_i^0}$.  Then, we set
  \begin{equation}\label{eq:140}
    \lambda := \frac{4}{n} \left( \rho^{n/4} - 1 \right).
  \end{equation}
  Clearly $\psi_i(\lambda) = \rho g_i^0 = g_i$ for $i = 0, 1, 2,
  \dots$.  Moreover, we claim that $\lambda \in L^2(M, g_0^0)$ and
  hence we can find a sequence $\{\lambda_l\}$ of bounded, measurable
  functions on $M$ that converge in $L^2(M, g_0^0)$ to $\lambda$ as $l
  \rightarrow \infty$.  That $\lambda \in L^2(M, g_0^0)$ follows from
  two facts.  First, $\rho \in L^{n/2}(M, g_0^0)$, implying that
  $\rho^{n/4} \in L^2(M, g_0^0)$.  Second, finite volume of $g_0^0$
  implies that the constant function $1 \in L^2(M, g_0^0)$ as well.

  Since $\rho > 0$, we have $\lambda > - \frac{4}{n}$.  Thus, we can
  choose our $\lambda_l$ such that $\lambda_l > -\frac{4}{n}$.  We
  then have, by Lemma \ref{lem:12}, that
  \begin{equation}\label{eq:19}
    d(\psi_0(\lambda_l), g_0) = d(\psi_0(\lambda_l), \psi_0(\lambda)) \leq
    \sqrt{n} \norm{\lambda - \lambda_l}_{g_0^0},
  \end{equation}
  and so by our choice of $\lambda_l$,
  \begin{equation}\label{eq:21}
    \lim_{l \rightarrow \infty} d(\psi_0(\lambda_l), g_0) \leq \lim_{l
      \rightarrow \infty} \sqrt{n} \norm{\lambda - \lambda_l}_{g_0^0} = 0.
  \end{equation}

  It is possible to choose each $\lambda_l$ such that $\lambda_l(x)
  \leq \lambda(x)$ for all $x \in M$.  Define $\rho_l$ to be the
  function such that $\psi_0(\lambda_l) = \rho_l g_0^0$, and set
  $\sigma_l := \rho_l \xi$.  Thus,
  \begin{equation*}
    \rho_l =
    \left(
      1 + \frac{n}{4} \lambda_l
    \right)^{4/n} \leq \left(
      1 + \frac{n}{4} \lambda
    \right)^{4/n} = \rho,
  \end{equation*}
  and since $g_i^0 = \xi g_i$, we have $\rho_l g_i^0 = \sigma_l g_i$
  for $i = 0, 1, 2, \dots$.  But since $0 < \rho_l \leq \rho$ and $0 <
  \xi = \rho^{-1}$, we see that $0 < \sigma_l \leq 1$ for all $l \in
  \N$.  Thus, by Lemma \ref{lem:8}, $\psi_k(\lambda_l) = \sigma_l g_k
  \xrightarrow{\Theta_M} \sigma_l g_0 = \psi_0(\lambda_l)$, as $k
  \rightarrow \infty$, for each $l \in \N$.  Additionally, since
  $\sigma_l g_0 = \rho_l g_0^0$ is bounded (because each $\lambda_l$
  is bounded by assumption), Proposition \ref{prop:3} gives that
  \begin{equation}\label{eq:17}
    \lim_{k \rightarrow \infty} d(\psi_k(\lambda_l),
    \psi_0(\lambda_l)) = 0
  \end{equation}
  for each $l \in \N$.
  
  Now, as above, Lemma \ref{lem:12} gives
  \begin{equation}\label{eq:29}
    d(\psi_k(\lambda_l), g_k) = d(\psi_k(\lambda_l), \psi_k(\lambda))
    \leq \sqrt{n} \norm{\lambda - \lambda_l}_{g_k^0}.
  \end{equation}
  Since $g_k \xrightarrow{\Theta_M} g_0$ and $\xi \leq 1$, Lemma
  \ref{lem:8} implies that $g_k^0 = \xi g_k \xrightarrow{\Theta_M} \xi
  g_0 = g_0^0$.  Thus, Lemma \ref{lem:9} yields
  \begin{equation}\label{eq:20}
    \lim_{k \rightarrow \infty} d(\psi_k(\lambda_l), g_k)
    \leq \lim_{k \rightarrow \infty} \sqrt{n} \norm{\lambda -
      \lambda_l}_{g_k^0} = \sqrt{n} \norm{\lambda -
      \lambda_l}_{g_0^0}.
  \end{equation}
  (We have used (\ref{eq:29}) to obtain the inequality and Lemma
  \ref{lem:9} to obtain the equality.)

  This gives us all the pieces we need to estimate $d(g_k, g_0)$ using
  the triangle inequality.  By (\ref{eq:20}), we can choose $k$ large
  enough that
  \begin{equation*}
    d(g_k, \psi_k(\lambda_l))
    \leq \sqrt{n} \norm{\lambda -
      \lambda_l}_{g_0^0} + \frac{\epsilon}{4}.
  \end{equation*}
  By (\ref{eq:17}), we can also choose $k$ large enough that
  \begin{equation*}
    d(\psi_k(\lambda_l), \psi_0(\lambda_l)) < \frac{\epsilon}{4}.
  \end{equation*}
  Finally, using (\ref{eq:21}), we can choose $l$ large enough that
  \begin{equation*}
    \sqrt{n} \norm{\lambda - \lambda_l}_{g_0^0} < \frac{\epsilon}{4}
    \qquad \textnormal{and} \qquad d(\psi_0(\lambda_l), g_0) < \frac{\epsilon}{4}.
  \end{equation*}
  Putting together these four inequalities completes the proof of the
  theorem.
\end{proof}

In view of Theorem \ref{thm:5}, we get the following immediate
corollary, which is perhaps of independent interest.  It will also be
useful in the next subsection.

\begin{corollary}\label{cor:1}
  Let $g_k, g_0 \in \Mf$, and let $g_k \xrightarrow{d} g_0$.  Then
  \begin{equation*}
    \left(
      \frac{\mu_{g_k}}{\mu_g}
    \right) \xrightarrow{L^1(M, g)}
    \left(
      \frac{\mu_{g_0}}{\mu_g}
    \right).
  \end{equation*}
  Equivalently, $\mu_{g_k}$ converges uniformly to $\mu_{g_0}$.
\end{corollary}

At this point, we wish to move on to studying the relation between the
completions of $(\M, d)$ and $(\M, \TM)$, which amounts to
investigating which sequences are Cauchy with respect these metrics,
and when two Cauchy sequences are equivalent.  But because pointwise,
$\overline{(\M, d)}$ and $\overline{(\M, \TM)}$ are in bijection with
one another (through their respective bijections with $\Mfhat$), we
can avoid direct considerations of these issues.  In fact, there is
little argumentation yet remaining.

\begin{theorem}\label{thm:3}
  A sequence in $\{g_k\} \subset \M$ is $d$-Cauchy if and only if it
  is $\Theta_M$-Cauchy.
\end{theorem}
\begin{proof}
  That a $d$-Cauchy sequence is $\Theta_M$-Cauchy is immediate from
  Proposition \ref{prop:6}.  So we turn to the converse statement.

  Let $\{g_k\}$ be $\TM$-Cauchy, and let $g_0 \in \Mf$ be a
  representative of its $\TM$-limit in $\Mfhat$ as guaranteed by
  Theorem \ref{thm:12}.  Then $g_k \xrightarrow{\TM} g_0$, implying
  that $g_k \xrightarrow{d} g_0$.  But as a $d$-convergent sequence,
  $\{g_k\}$ is necessarily $d$-Cauchy, as was to be proved.
\end{proof}

\begin{theorem}\label{thm:13}
  There exist natural homeomorphisms $\overline{(\M, d)} \cong
  \overline{(\M, \TM)} \cong \Mfhat$.
\end{theorem}
\begin{proof}
  We know that a $d$-Cauchy sequence is $\TM$-Cauchy and vice versa,
  as well as that both completions can be identified with $\Mfhat$.
  We must still see that two Cauchy sequences $\{g_k^0\}$ and
  $\{g_k^1\}$ are $d$-equivalent (i.e., $\lim_{k \rightarrow \infty}
  d(g_k^0, g_k^1) = 0$) if and only if they are $\TM$-equivalent
  (i.e., $\lim_{k \rightarrow \infty} \TM(g_k^0, g_k^1) = 0$).  But
  assume $\lim_{k \rightarrow \infty} d(g_k^0, g_k^1) = 0$.  Since
  $\Mfhat$ is complete with respect to $d$ and $\{g_k^0\}$ and
  $\{g_k^1\}$ are $d$-Cauchy, they $d$-converge to some elements $g_0$
  and $g_1$, respectively.  By Theorem \ref{thm:2}, the sequences
  $\TM$-converge to $g_0$ and $g_1$ as well.  However, since $\lim_{k
    \rightarrow \infty} d(g_k^0, g_k^1) = 0$, we must have that $g_0 =
  g_1$, from which we conclude that $\lim_{k \rightarrow \infty}
  \TM(g_k^0, g_k^1) = 0$.  The converse statement is proved in exactly
  the same way.
\end{proof}

\subsection{Another characterization of convergence in $\M$}\label{sec:anoth-char-conv}

The last major result of the paper is another characterization of
convergence in $\M$ that does not require reference to either $d$ or
$\TM$.  We will state it after a brief lemma.  After showing the
convergence result using Theorem \ref{thm:2}, we use it to prove the
discontinuity of various geometric quantities on $\M$.

\begin{lemma}\label{lem:15}
  Let $\{g_k\} \subset \Mf$ $\Theta_M$-converge to $g_0 \in \Mf$.
  Then for each representative $g_0 \in [g_0]$, $\{g_k\}$ $\absgx{\, \cdot \,}$-converges to $g_0$ in measure on $M
  \setminus D_{\{g_k\}} = M \setminus X_{g_0}$.
\end{lemma}
\begin{proof}
  Let $\delta, \epsilon > 0$ be given; we must find $k_0 \in \N$ such
  that $k \geq k_0$ implies that if
  \begin{equation*}
    Y_k^\delta := \left\{
      x \in M \midmid \absgx{g_0(x) - g_k(x)} \geq \delta
    \right\},
  \end{equation*}
  then $\Vol(Y_k^\delta \setminus X_{g_0}, g) < \epsilon$.
  
  For $\zeta, \tau > 0$, we define (as in Lemma \ref{lem:7})
  \begin{equation*}
    \Matx^{\zeta, \tau} :=
    \left\{
      a \in \Matx \midmid \det A \geq \zeta,\ \abs{a_{ij}}
      \leq \tau\ \textnormal{for all}\ 1 \leq i, j \leq n
    \right\}.
  \end{equation*}
  Choose $\zeta$ small enough and $\tau$ large enough that if
  \begin{equation*}
    E^{\zeta, \tau} :=
    \left\{
      x \in M \midmid g_0(x) \in \Matx^{\zeta, \tau}
    \right\},
  \end{equation*}
  then $\Vol(M \setminus(X_{g_0} \cup E^{\zeta, \tau}), g) < \epsilon
  / 2$.

  Furthermore, let $\eta$ be defined as in Lemma \ref{lem:7}.  By Theorem
  \ref{thm:4}, $\{g_k\}$ $\tgx$-converges to $g_0$ in measure, so we
  can find $k_0 \in \N$ such that $k \geq k_0$ implies that for
  \begin{equation*}
    Z_k^\delta := \left\{
      x \in E^{\zeta, \tau} \midmid \tgx(g_k(x), g_0(x) \geq \eta(\delta)
    \right\},
  \end{equation*}
  $\Vol(Z_k^\delta, g) < \epsilon / 2$.  On the other hand, we have by
  definition that
  \begin{equation*}
    \absgx{g_0(x) - g_k(x)} < \delta
  \end{equation*}
  for $x \in E^{\zeta, \tau} \setminus Z_k^\delta$, implying that
  $Y^\delta_k \cap E^{\zeta,\tau} \subseteq Z^\delta_k$, so
  \begin{equation*}
    \Vol(Y_k^\delta \setminus X_{g_0}, g) \leq \Vol(M \setminus (X_{g_0} \cup
    E^{\zeta, \tau}), g) + \Vol(Z_k^\delta, g) < \epsilon,
  \end{equation*}
  as was to be shown.
\end{proof}

\begin{theorem}\label{thm:1}
  Say $\{g_k\} \in \Mf$ and $g_0 \in \M$.  Then $g_k \xrightarrow{d}
  g_0$ if and only if
  \begin{enumerate}
  \item \label{item:1} $g_k \rightarrow g_0$ in measure,
  \end{enumerate}
  and additionally one of the following conditions holds:
  \begin{enumerate}
  \item[2a.] \label{item:2} $\mu_k$ converges to $\mu$ uniformly; or
  \item[2b.] \label{item:3}
    \begin{equation*}
      \left(
        \frac{\mu_{g_k}}{\mu_g}
      \right) \xrightarrow{L^1(M, g)}
      \left(
        \frac{\mu_{g_0}}{\mu_g}
      \right).
    \end{equation*}
  \end{enumerate}
\end{theorem}

\begin{remark}\label{rmk:3}
  Recall that by our terminology (cf.~Definition \ref{dfn:1}),
  condition \ref{item:1} means that $g_k$ $\absdot_{g(x)}$-converges to
  $g_0$ in measure.
\end{remark}

\begin{proof}[Proof of Theorem \ref{thm:1}]
  The equivalence of conditions 2a and 2b is given by Lemma
  \ref{lem:18}, so we will simply work with condition 2a.

  First, suppose that $g_k \xrightarrow{d} g_0$.  Then conditions 2a
  and 2b are immediately implied by Corollary \ref{cor:1}.
  Condition \ref{item:1} follows from applying Theorem \ref{thm:2}
  followed by Lemma \ref{lem:15}.  (Note that in this case,
  $D_{\{g_k\}} = X_{g_0} = \emptyset$, since $g_0 \in \M$.)

  Conversely, let conditions \ref{item:1} and 2a hold.  By Theorem
  \ref{thm:2}, it suffices to show that $g_k \xrightarrow{\TM} g_0$.

  Let $\epsilon > 0$ be given.  By an argument exactly analogous to
  that in the proof of the last lemma, condition \ref{item:1}
  implies that $g_k$ $\tgx$-converges to $g_0$ in measure.  Lemma
  \ref{lem:14} then allows us to take the measure it converges in to
  be $\mu_{g_0}$ instead of $\mu_g$ as usual.  This means that if
  \begin{equation*}
    E^\epsilon_k := \left\{
      x \in M \midmid \tgx(g_k(x), g_0(x))
      \geq \epsilon
    \right\},
  \end{equation*}
  then for $k$ large enough,
  \begin{equation}\label{eq:1}
    \Vol(E^\epsilon_k, g_0) < \epsilon.
  \end{equation}
  
  For each $k \in \N$, define $g^0_k := \chi(M \setminus E^\epsilon_k)
  g_k + \chi(E^\epsilon_k) g_0$.  Then $\theta^g_x(g_k^0(x), g_0(x)) =
  0$ for $x \in E^\epsilon_k$ and $g_k^0(x) = g_k(x)$ for $x \notin
  E^\epsilon_k$, so
  \begin{equation}\label{eq:28}
    \TM(g_k^0, g_0) = \integral{M \setminus
      E^\epsilon_k}{}{\theta^g_x(g_k(x), g_0(x))}{d \mu_g} < 
    \Vol(M, g) \cdot \epsilon.
  \end{equation}
  On the other hand, by Proposition \ref{prop:4}, we have that
  \begin{equation*}
    \TM(g_k, g_k^0) \leq C'(n)
      ( \Vol(E^\epsilon_k, g_k) + \Vol(E^\epsilon_k, g_0) ),
  \end{equation*}
  since $g_k^0$ and $g_k$ differ only on $E^\epsilon_k$, where $g_k^0
  = g_0$.  By condition 2a and \eqref{eq:1}, the above thus implies
  that for $k$ large enough, $\TM(g_k, g_k^0) \leq 3 C'(n) \cdot
  \epsilon$.  Thus, by \eqref{eq:28} and the triangle inequality, we
  conclude
  \begin{equation*}
    \TM(g_k, g_0) \leq (\Vol(M, g) + 3 C'(n)) \cdot \epsilon.
  \end{equation*}
  Since $\epsilon$ was arbitrary, we have shown the desired result.
\end{proof}

This characterization of convergence is the most useful in practice,
essentially since it is so weak and easy to check.  Unfortunately,
this weakness is exactly what causes so many problems when attempting
to study the pseudometric induced on $\M / \D$ by $d$ (see the
Introduction).  The first question that one would ask in this context
is whether $d$ induces a metric space structure on $\M / \D$.  If this
is not true, then we can find metrics $g_0$ and $g_1$ in separate
diffeomorphism orbits and a sequence $\{ \varphi_k \} \subset \D$ such
that
\begin{equation}\label{eq:35}
  \lim_{k \rightarrow \infty} d(\varphi_k^* g_0, g_1) = 0.
\end{equation}
It seems very difficult to obtain an obstruction to this situation
given the characterization of convergence in Theorem \ref{thm:1}.
(The only one that one we know of can be immediately read off---if
$\Vol(M, g_0) \neq \Vol(M, g_1)$, then (\ref{eq:35}) cannot occur
thanks to Lemma \ref{lem:3}.)  In particular, if some
diffeomorphism-invariant geometric data were continuous with respect
to $d$, then $d$ would separate $\D$-orbits with varying data.
However, as the following collection of examples shows, the most
obvious geometric data is, in fact, discontinuous in a strong way.

To be precise, let $J$ define some geometric data, that is, a map $\M
\times M \rightarrow Y$, where $Y$ is some metric space with metric
$\delta$.  (For example, if $J$ is scalar curvature, then $J$ maps
into $\R$.)  We say that $J$ is \emph{continuous in measure} at $g_0
\in \M$ if for all $g_k \xrightarrow{d} g_0$ and all $\epsilon > 0$,
we have
\begin{equation*}
  \lim_{k \rightarrow \infty} \mu_g
  \left(
    \left\{
      p \in M \midmid \delta(J(g_k, p), J(g_0, p)) \geq \epsilon
    \right\}
  \right) = 0.
\end{equation*}

\begin{example}\label{eg:1}
  Let $M = T^2$, the two-dimensional torus, with its standard chart.
  (We take this to be the rectangle $[-1,1] \times [-1,1]$, in the
  $xy$-plane $\R^2$, with opposite edges identified.)  Let $g_0$ be the
  flat metric induced on $T^2$ via restriction of the Euclidean metric
  on $\R^2$ to this chart.  Then the following sequences $\{g_k\}$
  show that basic geometric data are discontinuous in measure at
  $g_0$.

  \emph{Curvature.}  If $J$ is any type of curvature, then it is clear
  from Theorem \ref{thm:1} that $J$ is discontinuous in measure on
  $\M$.

  \emph{Distance function.} Let $J$ be the distance function of the
  metric, i.e., $J(\tilde{g}, p) = d(p, \cdot) \in C^0(M)$.  For each
  $s \in (0, \frac{1}{2}]$, let $f_s$ be a smooth function defined on
  $[-1,1]$ such that $1 \leq f_s(t) \leq s^{-4}$ for all $t \in
  [-1,1]$, $f_s(s) = f_s(-s) = s^{-4}$, and $f_s(2s) = f_s(-2s) = 1$.
  Finally, let $g_k$ be the sequence given by
  \begin{equation*}
    g_k(x,y) :=
    \begin{cases}
      \begin{pmatrix}
        1 & 0 \\
        0 & 1
      \end{pmatrix}, & x \in [-1, -2/k] \cup [2/k, 1], \\
      \begin{pmatrix}
        f_{1/k}(x) & 0 \\
        0 & f_{1/k}^{-1}(x)
      \end{pmatrix}, & x \in (-2/k, -1/k) \cup (1/k, 2/k), \\
      \begin{pmatrix}
        k^4 & 0 \\
        0 & k^{-4}
      \end{pmatrix}, & x \in [-1/k, 1/k].
    \end{cases}
  \end{equation*}
  It is not hard to see from Theorem \ref{thm:1} that $g_k
  \xrightarrow{d} g_0$.  Furthermore, since $g_k(x,y)$ is constant in
  $y$, we see that the geodesics connecting points with equal
  $y$-coordinates are horizontal lines.  On the other hand, one also
  easily computes that the length, with respect to $g_k$, of
  horizontal lines passing all the way through the cylindrical region
  $\{(x,y) \mid x \in [-1/k,1/k] \}$ is unbounded as $k \rightarrow
  \infty$.  Thus the distance function is discontinuous in measure at
  $g_0$---geometrically, the torus converges a cylinder with two
  infinitely long cusps as ends.  (This is provided the functions
  $f_s$ are chosen ``well''.  The convergence can be taken to be,
  e.g., pointed Gromov--Hausdorff convergence.  See \cite[\S
  3.B]{gromov-metric-2007}.)  Note that more specifically, we have
  shown that the distance function is not ``upper semicontinuous in
  measure'' on $\M$.

  \emph{Diameter.}  Let $J(\tilde{g}, p)$ be the diameter, i.e.,
  $J(\tilde{g}, p) = \diam(M, \tilde{g})$ independently of $p$.
  Clearly, the above example shows that this is also discontinuous in
  measure at $g_0$.

  \emph{Injectivity radius.}  Let $J(\tilde{g}, p) =
  \inj_{\tilde{g}}(p)$.  Fix $k \in \N$, and define a region $E_k
  \subset T^2$ by
  \begin{equation*}
    E_k := \{ (x,y) \mid x \in [-3/4, 3/4],\ y \in [-1/k, 1/k] \}.
  \end{equation*}
  Note that $\Vol(E_k, g_0) = 3/k$.  Let $U_k$ be an open domain with
  $E_k \subset U_k$ and satisfying $\Vol(U_k, g_0) \leq 4/k$.  Now,
  choose any metric $g_k \in \M$ such that $g_k(p) = g_0$ for $p
  \notin U_k$, and for $p \in E_k$,
  \begin{equation*}
    g_k(p) = 
    \begin{pmatrix}
      k^{-1} & 0 \\
      0 & k
    \end{pmatrix}.
  \end{equation*}
  Furthermore, we assume that $g_k$ is chosen such that $(g_k)_{11}(p)
  \leq 1$ and $\mu_{g_k}(p) = \mu_{g_0}(p)$ for all $p \in U_k$.  Then
  one can compute that, with respect to $g_k$, the length of the
  closed curve $\gamma$ given by $\gamma(t) = (t, 0)$, for $t \in
  [-1,1]$, satisfies
  \begin{equation*}
    \limsup_{k \rightarrow \infty} L_{g_k}(\gamma) \leq \frac{1}{2}.
  \end{equation*}
  But from this, it is not hard to infer that $\limsup_{k \rightarrow
    \infty} \inj_{g_k}(p) < 1$ for $p$ on a set of positive
  $\mu_g$-measure.  Since it is also clear by Theorem \ref{thm:1} that
  $g_k \xrightarrow{d} g_0$, this shows that the injectivity radius is
  discontinuous in measure at $g_0$.  Note that this example also
  shows that the distance function is not even ``lower semicontinuous
  in measure''.
\end{example}

We remark that these examples do not give the situation of
(\ref{eq:35}), since the metrics $g_k$ in each are clearly mutually
non-isometric.  In fact, to date we do not have a single example where
we can compute the $d$-distance between two elements of $\M / \D$ with
equal total volumes, or even estimate this distance away from zero.

\bibliography{main}
\bibliographystyle{hamsplain}

\end{document}